\documentclass[11pt]{amsart}
\usepackage[english]{babel}
\usepackage{amssymb,amscd}

\usepackage[all]{xy}
\usepackage{color}

\usepackage{euscript}
\usepackage{amsmath}
\usepackage{ifpdf}

\usepackage{amsfonts}
\usepackage{mathrsfs}

\usepackage{mathpazo}

\newtheorem{Theorem}{Theorem}[section]
\newtheorem{Lemma}[Theorem]{Lemma}
\newtheorem{Corollary}[Theorem]{Corollary}
\newtheorem{Proposition}[Theorem]{Proposition}

\newtheorem{Remark}[Theorem]{Remark}
\newtheorem{Example}[Theorem]{Example}

\newtheorem{Definition}[Theorem]{Definition}

\newtheorem{Question}[Theorem]{Question}
\newcommand{\M}{\mathfrak{m}}

\def\sqr#1#2{{\vcenter{\hrule height.#2pt
\hbox{\vrule width.#2pt height#1pt \kern#1pt \vrule width.#2pt}
\hrule height.#2pt}}}

\def\qed{\hspace*{\fill} $\square$}

\begin{document}

\title[Reduction numbers and regularity of blowup rings and modules]{On reduction numbers and Castelnuovo-Mumford regularity of blowup rings and modules}

\author{Cleto B.~Miranda-Neto}
\address{Departamento de Matemática, Universidade Federal da
Paraíba, 58051-900 João Pessoa, PB, Brazil.}
\email{cleto@mat.ufpb.br}

\author{Douglas S.~Queiroz}
\address{Departamento de Matemática, Universidade Federal da
Paraíba, 58051-900 João Pessoa, PB, Brazil.}
\email{douglassqueiroz0@gmail.com}

\thanks{
Corresponding author: C. B. Miranda-Neto.}

\subjclass[2010]{Primary: 13C05, 13A30, 13C13; Secondary: 13H10, 13D40, 13D45.}
\keywords{Castelnuovo-Mumford regularity, reduction number, blowup algebra, Ratliff-Rush closure, Ulrich ideal}

\begin{abstract} We prove new results on the connections between reduction numbers and the Castelnuovo-Mumford regularity of blowup algebras and blowup modules, the key basic tool being the operation of Ratliff-Rush closure. First, we answer in two particular cases a question of  M. E. Rossi, D. T. Trung, and N. V. Trung about Rees algebras of ideals in two-dimensional Buchsbaum local rings, and we even ask whether one of such situations always holds. In another theorem we generalize a result of A. Mafi on ideals in two-dimensional Cohen-Macaulay local rings, by extending it to arbitrary dimension (and allowing for the setting relative to a Cohen-Macaulay module). We derive a number of applications, including a characterization of (polynomial) ideals of linear type, progress on the theory of generalized Ulrich ideals, and improvements of results by other authors. 

%By detecting a normal Ulrich ideal in the ring %of a rational double point, we also provide a %negative answer to a question of A. Corso, C. %Polini, and M. E. Rossi about normal ideals in %Gorenstein local rings.

\end{abstract}

\maketitle

\vspace{-0.1in}

%\centerline{\it To the memory of Professor Wolmer %Vasconcelos, mentor of generations.}

\centerline{\it Dedicated to the memory of Professor Wolmer Vasconcelos,} 
\centerline{\it mentor of generations of commutative algebraists.}

\medskip

\section{Introduction}

This paper aims at further investigating the interplay between reduction numbers, eventually taken relative to a given module, and the Castelnuovo-Mumford regularity of Rees and associated graded algebras and modules. The basic tool is the so-called Ratliff-Rush closure, a well-known operation which goes back to the classical work \cite{Ratliff-Rush}, and also  \cite{Heinzer-Johnson-Lantz-Shah} for the relative case. These topics have been studied for decades and the literature about them is extensive; we mention, e.g., \cite{AHT}, \cite{Brodmann-Linh}, \cite{CPR},
\cite{D-H}, \cite{H-P-T}, \cite{J-U}, \cite{Linh}, \cite{Mafi}, \cite{Marley2}, \cite{Pol-X}, \cite{Puthenpurakal1}, \cite{Rossi-Dinh-Trung}, \cite{R-T-V}, \cite{Strunk}, \cite{Trung}, \cite{Vasc}, \cite{Wu}, \cite{Zamani}.

Let us briefly describe our two main motivations. First, an interesting question raised by Rossi, Trung, and Trung \cite{Rossi-Dinh-Trung} concerning the Castelnuovo-Mumford regularity of the Rees algebra of a zero-dimensional ideal in a 2-dimensional Buchsbaum local ring, and second, a result due to Mafi \cite{Mafi} on the Castelnuovo-Mumford regularity of the associated graded ring of a zero-dimensional ideal in a 2-dimensional Cohen-Macaulay local ring, both connected to the reduction number of the ideal. For the former (see Section \ref{R-T-T}) we provide answers in a couple of special cases, and we even ask whether one of the situations always occurs. For the latter (see Section \ref{gen-of-Mafi}), we are able to establish a generalization which extends dimension $2$ to arbitrary dimension and in addition takes into account the more general context of blowup modules (relative to any given Cohen-Macaulay module of positive dimension).

We then provide two sections of applications. In Section \ref{app} we show, for instance, that under appropriate conditions the Castelnuovo-Mumford regularity of Rees modules is not affected under regular hyperplane sections, and moreover, by investigating the role of postulation numbers with a view to the regularity of Rees algebras, we derive an improvement (by different arguments) of a result originally obtained in \cite{Hoa} on the independence of reduction numbers. Still in Section \ref{app} we provide an application regarding ideals of linear type (i.e., ideals for which the Rees and symmetric algebras are isomorphic) in a standard graded polynomial ring over an infinite field. Precisely, if $I$ is a zero-dimensional ideal with reduction number zero, then $I$ is of linear type (the converse is easy and well-known). Finally, in Section \ref{Ulrich}, we establish new facts on the theory of generalized Ulrich ideals introduced in \cite{Goto-Ozeki-Takahashi-Watanabe-Yoshida}. Given such an ideal in a Cohen-Macaulay local ring (of positive dimension), we determine the Castelnuovo-Mumford regularity of its blowup algebras and describe explicitly its Hilbert-Samuel polynomial and postulation number, which as far as we know are issues that have not been considered in the literature. We also correct a result from \cite{Mafi2} and improve results of \cite{Hoa}, \cite{Huneke} and \cite{Itoh} (where extra hypotheses were required), regarding a well-studied link between low reduction numbers and Hilbert functions. In addition, by detecting a normal Ulrich ideal in the ring of a rational double point, we give a negative answer to the 2-dimensional case of a question posed over fifteen years ago in \cite{CPR}.

\medskip

\noindent {\it Conventions}. Throughout the paper, by {\it ring} we mean commutative, Noetherian, unital ring. If $R$ is a ring, then by a {\it finite} $R$-module we mean a finitely generated $R$-module.

\section{Ratliff-Rush closure}

%In their investigation of reductions of ideals, %Ratliff and Rush \cite{Ratliff-Rush} introduced 

The concept of the so-called \textit{Ratliff-Rush closure} $\widetilde{I}$
of a given ideal $I$ in a ring $R$ first appeared in the investigation of reductions of ideals developed in \cite{Ratliff-Rush}. Precisely,
$$\widetilde{I} \, = \, \bigcup_{n \geq 1}\,I^{n+1} : I^{n}.$$
This is an ideal of $R$ containing $I$ which refines the integral closure of $I$, so that $\widetilde{I}=I$ whenever $I$ is integrally closed. Now suppose $I$ contains a regular element (i.e., a non-zero-divisor on $R$). Then $\widetilde{I}$ is the largest ideal that shares with $I$ the same sufficiently high powers. Moreover, if $M$ is an $R$-module then \cite[Section 6]{Heinzer-Johnson-Lantz-Shah} defined the \textit{Ratliff-Rush closure of $I$ with respect to $M$} as
$$\widetilde{I_{M}} \, = \, \bigcup_{n \geq 1}\,I^{n + 1}M :_{M} I^{n} \, = \, \{m \in M \, \mid \, I^{n}m \subseteq I^{n + 1}M \ \mathrm{for \ some \ n\geq 1} \}.$$ Clearly, this retrieves the classical definition by letting $M = R$. Furthermore, $\widetilde{I_{M}}$ is an $R$-submodule of $M$, and it is easy to see that $IM\subseteq \widetilde{I}M  \subseteq \widetilde{I_{M}}$. If the equality $IM=\widetilde{I_{M}}$ holds, the ideal $I$ is said to be {\it Ratliff-Rush closed with respect to $M$}; in case $M=R$, we simply say that $I$ is {\it Ratliff-Rush closed} (some authors also use the expression {\it Ratliff-Rush ideal}).

\begin{Lemma} $($\cite[Proposition 2.2]{Naghipour}\label{propnaghipour}$)$ \label{lemmanaghipour} Let $R$ be a ring and $M$ be a non-zero finite $R$-module. Assume that $I$ is an ideal of $R$ containing an $M$-regular element {\rm (}i.e., a non-zero-divisor on $M${\rm )} and such that $IM \neq M$ {\rm (}e.g., if $R$ is local{\rm )}. Then the following conditions hold:

	\begin{itemize}
		\item[(a)] For an integer $n \geq 1$, $\widetilde{I^{n}_{M}}$ is the eventual stable value of the increasing sequence
		$$(I^{n + 1}M :_{M} I) \, \subseteq \, (I^{n + 2}M :_{M} I^{2}) \, \subseteq \, (I^{n + 3}M :_{M} I^{3}) \, \subseteq \, \cdots$$
		\item[(b)] $\widetilde{I_{M}} \supseteq \widetilde{I^{2}_{M}} \supseteq \widetilde{I^{3}_{M}} \supseteq \cdots \supseteq \widetilde{I^{n}_{M}} = I^{n}M$ for all $n\gg 0$, and hence $I^{n}$ is Ratliff–Rush closed with respect to $M$.
	\end{itemize}
\end{Lemma}

Note that, if the ideal $I$ contains an $M$-regular element, Lemma \ref{propnaghipour}(b) enables us to define the number $$s^{*}(I,M) \, := \, \mathrm{min}\,\{n \in \mathbb{N} \, \mid \, \widetilde{I^{i}_{M}} = I^{i}M \ \mathrm{for \ all} \ i \geq n \},$$ which we simply write $s^{*}(I)$ if $M=R$. Since the equality $\widetilde{I^{i}_{M}} = I^{i}M$ holds trivially for $i=0$, we have that $s^{*}(I,M)\geq 0$ if and only if $s^{*}(I,M)\geq 1$. Thus we can establish that, throughout the entire paper, $$s^{*}(I,M) \, \geq \, 1.$$

\begin{Definition}\rm
	Let $R$ be a ring, $I$ an ideal of $R$, and $M$ an $R$-module. We say that $x \in I$ is an $M$-\textit{superficial element of $I$} if there exists $c \in \mathbb{N}$ such that $(I^{n + 1}M :_{M} x)  \cap  I^{c}M  =  I^{n}M$, for all $n \geq c$. If $x$ is $R$-superficial, we simply say that $x$ is a \textit{superficial element of $I$}. A sequence $x_{1}, \ldots , x_{s}$ in $I$ is said to be an \textit{$M$-superficial sequence of $I$} if $x_1$ is an $M$-superficial element of $I$ and, for all $i = 2, \ldots , s$, the image of $x_{i}$ in $I/(x_{1}, \ldots , x_{i - 1})$ is an $M/(x_{1}, \ldots , x_{i - 1})M$-superficial element of  $I/(x_{1}, \ldots , x_{i - 1})$.
\end{Definition}

If $R$ is local, the above concept allows for an alternative, useful expression for $s^{*}(I,M)$. Indeed, in this case, \cite[Corollary 2.7]{Puthenpurakal1} shows that
$$s^{*}(I,M) \, = \, \mathrm{min}\,\{n \in \mathbb{N} \, \mid \, I^{i + 1}M :_{M} x = I^{i}M \ \mathrm{for \ all} \ i \geq n \},$$
for any given $M$-superficial element $x$ of $I$.

\begin{Lemma}$($\cite[Lemma 8.5.3]{Swanson-Huneke}\label{lemmaswanhun1}$)$
	Let $R$ be a ring, $I$ an ideal of $R$, $M$ a finite $R$-module, and $x \in I$ an $M$-regular element. Then, $x$ is an $M$-superficial element of $I$ if and only if $I^{n}M\colon_Mx = I^{n - 1}M$ for all $n\gg 0$.
\end{Lemma}

\begin{Lemma}$($\cite[Lemma 8.5.4]{Swanson-Huneke}\label{lemmaswanhun}$)$
	Let $R$ be a ring, $I$ an ideal of $R$, and $M$ a finite $R$-module. Assume that $\bigcap_{n} I^{n}M = 0$ and that $I$ contains an $M$-regular element. Then every $M$-superficial element of $I$ is $M$-regular.
\end{Lemma}

The condition $\bigcap_{n} I^{n}M = 0$ is satisfied whenever $R$ is a local ring, as guaranteed by the well-known Krull's intersection theorem, so in order to use Lemma \ref{lemmaswanhun} we just need to find an $M$-regular element in $I$.

Lemma \ref{lemmaswanhun1} and Lemma \ref{lemmaswanhun} put us in a position to prove Proposition \ref{regelem} below, which, according to \cite[Lemma 2.2]{Mafi}, is known when $M = R$. Moreover, in our general context, part (b) has been stated in \cite[Theorem 2.2(2)]{Puthenpurakal1}. However, the proofs are omitted in both situations, so for completeness we provide them here.

\begin{Proposition} \label{regelem}
	Let $R$ be a local ring, $M$ a finite $R$-module, and $I$ an ideal of $R$ containing an $M$-regular element. Then, for all $m \in \mathbb{N}$,
	
	\begin{itemize}
		\item[(a)] $\widetilde{I^{m + 1}_{M}} :_{M} I = \widetilde{I^{m}_{M}};$
		
		\item[(b)] $\widetilde{I^{m + 1}_{M}} :_{M} x = \widetilde{I^{m}_{M}}$, for every $M$-superficial element $x$ of $I$.
		\end{itemize}
\end{Proposition}	
\begin{proof} Fix an arbitrary $m \in \mathbb{N}$. By Lemma \ref{lemmanaghipour}(a), we can write $\widetilde{I^{m }_{M}} = I^{m + j}M :_{M} I^{j} = I^{m + j + 1}M :_{M} I^{j + 1}$ and $\widetilde{I^{m + 1}_{M}} = I^{(m + 1) + (j - 1)}M :_{M} I^{j - 1} = I^{m + j + 1}M :_{M} I^{j},$ for all $j \gg 0$.  In particular, we get $$\widetilde{I^{m + 1}_{M}} :_{M} I \, = \, (I^{m + j + 1}M :_{M} I^{j}) :_{M} I \, = \, I^{m + j + 1}M :_{M} I^{j + 1} \, = \, \widetilde{I^{m }_{M}},$$ which proves (a). Now, let $x$ be an $M$-superficial element of $I$. By Lemma \ref{lemmaswanhun}, $x$ is $M$-regular, and hence, in view of Lemma \ref{lemmaswanhun1}, we can choose $j$ such that $I^{m + j + 1}M :_{M} x = I^{m + j}$. It follows that
$$\widetilde{I^{m + 1}_{M}} :_{M} x \, = \, (I^{m + j + 1}M :_{M} I^{j}) :_{M} x \, = \, (I^{m + j + 1}M :_{M} x) :_{M} I^{j} \, = \, I^{m + j}M :_{M} I^{j} \, = \, \widetilde{I^{m}_{M}},$$ thus showing (b). \qed
\end{proof}

\medskip

As a consequence of item (a) we derive the following result, which will be a key tool in Subsection
\ref{RTT} in the case $M=R$. Part (b) of Proposition \ref{regelem} will be used in other sections (e.g., in the proof of Theorem \ref{Mafigene}).

\begin{Proposition}\label{s*conjec}
	Let $R$ be a local ring and $M$ a non-zero finite $R$-module. Let $I$ be an ideal of $R$ containing an $M$-regular element. Then $$s^{*}(I, M) \, = \, \mathrm{min}\,\{m \geq 1 \, \mid \, I^{n + 1}M :_{M} I = I^{n}M \ \ \mathit{for \ all} \ \ n \geq m \}.$$
\end{Proposition}
\begin{proof} Write $s = s^{*}(I,M)$ and $t = \mathrm{min}\{m \geq 1 \mid I^{n + 1}M :_{M} I = I^{n}M \ \mathrm{for \ all} \ n \geq m \}$. Let $n \geq s$. From the definition of $s$ we get $ \widetilde{I^{n}_{M}}  =  I^{n}M$ and $\widetilde{I^{n + 1}_{M}}   =  I^{n + 1}M$. On the other hand, Proposition \ref{regelem}(a) gives $\widetilde{I^{n + 1}_{M}} :_{M} I   =  \widetilde{I^{n}_{M}}$. Hence,
$I^{n + 1}M :_{M} I  =  I^{n}M$ for all $n\geq s$, which forces $s \geq t$. Suppose by way of contradiction that $s > t$. Then $s-1\geq t\geq 1$ and, using Proposition \ref{regelem}(a) again, we obtain
$$\widetilde{I^{s - 1}_{M}} \, = \, \widetilde{I^{s}_{M}} :_{M} I \, = \, I^{s}M :_{M} I \, = \, I^{s - 1}M,$$ 
thus contradicting the minimality of $s$. We conclude that $s=t$, as needed.\qed
\end{proof}

\section{First results and a question of Rossi-Trung-Trung}\label{R-T-T}

Let $S = \bigoplus_{n \geq 0}S_{n}$ be a finitely generated standard graded algebra over a ring $S_{0}$. As usual, by {\it standard} we mean that $S$ is generated by $S_1$ as an $S_0$-algebra. We write $S_{+} = \bigoplus_{n \geq 1}S_{n}$ for the ideal generated by all elements of $S$ of positive degree. Fix a finite graded $S$-module $N\neq 0$. A sequence $y_{1}, \ldots, y_{s}$ of homogeneous elements of $S$ is said to be an \textit{$N$-filter regular sequence} if $$y_{i} \notin P, \ \ \ \ \ \forall \ P \in \textrm{Ass}_{S}\left(\frac{N}{(y_{1}, \ldots, y_{i - 1})N}\right)\setminus V(S_{+}), \ \ \ \ \ i = 1, \ldots, s.$$ 

Now, for a graded $S$-module $A=\bigoplus_{n \in {\mathbb Z}}A_{n}$ satisfying $A_n=0$ for all $n\gg 0$, we let $a(A) = \textrm{max}\{n\in {\mathbb Z} \ | \ A_{n} \neq 0\}$ if $A\neq 0$, and $a(A)=-\infty$ if $A=0$. For an integer $j\geq 0$, we use the notation $$a_{j}(N) \, := \, a(H_{S_{+}}^{j}(N)),$$
where $H_{S_{+}}^{j}(-)$ stands for the $j$-th local cohomology functor with respect to the ideal $S_{+}$. It is known that $H_{S_{+}}^{j}(N)$ is a graded module with $H_{S_{+}}^{j}(N)_n=0$ for all $n\gg 0$; see, e.g., \cite[Proposition 15.1.5(ii)]{B-S}. Thus, $a_{j}(N)<+\infty$. We can now invoke one of the main numerical invariants studied in this paper.

\begin{Definition}\rm Maintain the above setting and notations. The \textit{Castelnuovo-Mumford regularity} of the $S$-module $N$ is defined as
$$\mathrm{reg}\,N \, := \, \mathrm{max}\{a_j(N) + j \, \mid \, j \geq 0\}.$$
\end{Definition}

It is well-known that $\mathrm{reg}\,N$ governs the complexity of the graded structure of $N$ and is of great significance in commutative algebra and algebraic geometry, for instance in the study of degrees of syzygies over polynomial rings. The literature on the subject is extensive; see, e.g.,   \cite{Bay-S}, \cite[Chapter 15]{B-S},  \cite{Eis-Go}, and \cite{Trung2}.

\subsection{First results} Let $I$ be an ideal of a ring $R$. The {\it Rees ring of $I$} is the blowup algebra ${\mathcal R}(I)  =  \bigoplus_{n \geq 0}I^{n}t^n$ (as usual we put $I^0=R$), which can be expressed as the standard graded subalgebra $R[It]$ of $R[t]$, where $t$ is an indeterminate over $R$. Now, for a finite $R$-module $M$, the {\it Rees module of $I$ relative to $M$} is the blowup module $${\mathcal R}(I, M) \, = \,  \bigoplus_{n \geq 0}\,I^{n}M,$$ which in particular is a finite graded module over ${\mathcal R}(I, R)={\mathcal R}(I)$. We will be interested in the relation between ${\rm reg}\,{\mathcal R}(I, M)$ and another key invariant that will be defined shortly. First we recall the following fact.

\begin{Lemma}$($\cite[Lemma 4.5]{Giral-Planas-Vilanova}$)$\label{aistheleast} Let $R$ be a ring, $I$ an ideal of $R$ and $M$ a finite $R$-module. Let $x_{1}, \ldots, x_{s}$ be elements of $I$. Then, $x_{1}t, \ldots, x_{s}t$ is an ${\mathcal R}(I, M)$-filter regular sequence if and only if
	\begin{equation}\label{GP-V}
	[(x_{1}, \ldots, x_{i - 1})I^{n}M :_{M} x_{i}] \, \cap \, I^{n}M \, = \, (x_{1}, \ldots, x_{i - 1})I^{n - 1}M
	\end{equation}
	for $i = 1, \ldots, s$ and all $n\gg 0$. 
	\end{Lemma}

\begin{Definition}\rm Let $I$ be a proper ideal of a ring $R$ and let $M$ be a non-zero finite $R$-module. An ideal $J \subseteq I$ is called a \textit{reduction of $I$ relative to $M$} if $JI^{n}M = I^{n + 1}M$ for some integer $n\geq 0$. Such a reduction $J$ is said to be \textit{minimal} if it is minimal with respect to inclusion. As usual, in the classical case where $M = R$  we say that $J$ is a \textit{reduction of $I$}. If $J$ is a reduction of $I$ relative to $M$, we define the \textit{reduction number of $I$ with respect to $J$ relative to $M$} as $${\rm r}_{J}(I, M) \, := \, \mathrm{min}\,\{m \in \mathbb{N} \, \mid \, JI^{m}M = I^{m + 1}M\},$$
and the \textit{reduction number of $I$ relative to $M$} as
$${\rm r}(I, M) \, := \, \mathrm{min}\,\{{\rm r}_{J}(I, M) \, \mid \, J \ \ \mathrm{is \ a \ minimal \ reduction \ of \ } I \mathrm{ \ relative \ to} \ M\}.$$ We say that
${\rm r}(I, M)$ is \textit{independent} if ${\rm r}_J(I,M) = {\rm r}(I, M)$ for every minimal reduction $J$ of $I$ relative to $M$. If $M = R$, we simply write ${\rm r}_{J}(I)$ (the \textit{reduction number of $I$ with respect to $J$}) and ${\rm r}(I)$ (the \textit{reduction number of $I$}).
\end{Definition}

\begin{Lemma} $($\cite[Proposition 4.6]{Giral-Planas-Vilanova}$)$\label{propleastnum}
	Let $R$ be a ring, $I$ an ideal of $R$ and $M$ a finite $R$-module. Let $J = (x_{1}, \ldots, x_{s})$ be a reduction of $I$ relative to $M$ such that $x_{1}t, \ldots, x_{s}t$ is an ${\mathcal R}(I, M)$-filter regular sequence. Then 
$$\mathrm{reg}\,{\mathcal R}(I, M) \, = \, \mathrm{min}\,\{\ell \geq {\rm r}_{J}(I,M) \, \mid \,  \mathit{equality \ in \ {\rm (\ref{GP-V})} \ holds \ for \ all} \ \ n \geq \ell + 1\}.$$
\end{Lemma}

The usefulness of the proposition below will be made clear in Remark \ref{reg<r} and Remark \ref{cotareg}. Notice that the case $M=R$ recovers \cite[Proposition 4.7(i)]{Trung}; in fact, the proof is essentially the same and we give it in a little more detail.

\begin{Proposition}\label{proptrung}
	Let $R$ be a  ring, $I$ an ideal of $R$ and $M$ a finite $R$-module. Let $J = (x_{1}, \ldots, x_{s})$ be a reduction of $I$ relative to $M$ such that
	$$[(x_{1}, \ldots, x_{i - 1})M :_{M} x_{i}] \, \cap \, I^{\ell + 1}M \, = \, (x_{1}, \ldots, x_{i - 1})I^{\ell}M,  \ \ \ \ \ i = 1, \ldots, s.$$
	for a fixed $\ell \geq {\rm r}_{J}(I,M)$. Then, for all $n \geq \ell + 1$,
	$$[(x_{1}, \ldots, x_{i - 1})M :_{M} x_{i}] \, \cap \, I^{n}M \, = \, (x_{1}, \ldots, x_{i - 1})I^{n - 1}M,  \ \ \ \ \ i = 1, \ldots, s.$$
\end{Proposition}
\begin{proof} We may take $n > \ell + 1$ as the case $n=\ell + 1$ holds by hypothesis. Since $I^{n}M = JI^{n - 1}M= (x_1, \ldots, x_{s-1})I^{n - 1}M + x_sI^{n - 1}M$, and $(x_1, \ldots, x_{s-1})I^{n - 1}M\subseteq (x_1, \ldots, x_{s-1})M\subseteq (x_{1}, \ldots, x_{s - 1})M :_{M} x_{s}$, we can write 
		$$ [(x_{1}, \ldots, x_{s - 1})M :_{M} x_{s}] \cap I^{n}M = [(x_{1}, \ldots, x_{s - 1})M :_{M} x_{s}] \cap [ (x_1, \ldots, x_{s-1})I^{n - 1}M + x_sI^{n - 1}M]$$ 
		$$=(x_{1}, \ldots, x_{s - 1})I^{n - 1}M + \{[(x_{1}, \ldots, x_{s - 1})M :_{M} x_{s}] \cap x_sI^{n - 1}M\}$$
		$$=(x_{1}, \ldots, x_{s - 1})I^{n - 1}M  + x_{s}\{ [(x_{1}, \ldots, x_{s - 1})M :_{M} x_{s}^{2}] \cap I^{n - 1}M\}.$$
		Induction on $n$ yields $$[(x_{1}, \ldots, x_{s - 1})M :_{M} x_{s}] \cap I^{n - 1}M = (x_{1}, \ldots, x_{s - 1})I^{n - 2}M.$$ Thus,
		$$[(x_{1}, \ldots, x_{s - 1})M :_{M} x_{s}^{2}] \cap [I^{n - 1}M :_{M} x_{s}] =  \{[(x_{1}, \ldots, x_{s - 1})M :_{M} x_{s}] \cap I^{n - 1}M\} :_{M} x_{s}$$  $$=(x_{1}, \ldots, x_{s - 1})I^{n - 2}M :_{M} x_{s}
		\subseteq (x_{1}, \ldots, x_{s - 1})M :_{M} x_{s}.$$ Intersecting with $I^{n - 1}M$, we get $$[(x_{1}, \ldots, x_{s - 1})M :_{M} x_{s}^{2}] \cap I^{n - 1}M  \subseteq  [(x_{1}, \ldots, x_{s - 1})M :_{M} x_{s}] \cap I^{n - 1}M   = (x_{1}, \ldots, x_{s - 1})I^{n - 2}M.$$ Therefore, $$x_{s}\{[(x_{1}, \ldots, x_{s - 1})M :_{M} x_{s}^{2}] \cap I^{n - 1}M\}  \subseteq  x_{s}[(x_{1}, \ldots, x_{s - 1})I^{n - 2}M]  \subseteq  (x_{1}, \ldots, x_{s - 1})I^{n - 1}M.$$ 
		
		Putting these facts together, we obtain 
		$$[(x_{1}, \ldots, x_{s - 1})M :_{M} x_{s}] \cap I^{n}M = (x_{1}, \ldots, x_{s - 1})I^{n - 1}M.$$
		Now, for $i < s$, we can use induction on $n$ and on $i$ to obtain
		$$[(x_{1}, \ldots, x_{i - 1})M :_{M} x_{i}] \cap I^{n}M  =  \{[(x_{1}, \ldots, x_{i - 1})M :_{M} x_{i}] \cap I^{n - 1}M\} \cap I^{n}M$$  $$=  [(x_{1}, \ldots, x_{i - 1})I^{n - 2}M] \cap I^{n}M   \subseteq  [(x_{1}, \ldots, x_{i})M] \cap I^{n}M$$  $$\subseteq  [(x_{1}, \ldots, x_{i})M :_{M} x_{i + 1}] \cap I^{n}M=(x_{1}, \ldots, x_{i})I^{n - 1}M=(x_{1}, \ldots, x_{i-1})I^{n - 1}M+x_iI^{n-1}M.$$ It follows that 
		$$ [(x_{1}, \ldots, x_{i - 1})M :_{M} x_{i}] \cap I^{n}M = (x_{1}, \ldots, x_{i - 1})I^{n - 1}M + x_{i}\{ [(x_{1}, \ldots, x_{i - 1})M :_{M} x_{i}^{2}] \cap I^{n - 1}M\}.$$ Now the result follows similarly as in the case $i = s$. \qed
\end{proof}

\medskip 

Let us now introduce another fundamental blowup module. For an ideal $I$ in a ring $R$ and a finite $R$-module $M$, the {\it associated graded module of $I$ relative to $M$} is the blowup module $${\mathcal G}(I, M) \, = \, \bigoplus_{n \geq 0}\, \frac{I^{n}M}{I^{n + 1}M} \, = \, {\mathcal R}(I, M)\otimes_RR/I,$$ which is a finite graded module over the {\it associated graded ring} ${\mathcal G}(I) =  {\mathcal G}(I, R)$ of $I$, i.e.,  ${\mathcal G}(I)  = {\mathcal R}(I)\otimes_RR/I$. Notice that this algebra is standard graded over ${\mathcal G}(I)_0=R/I$. As usual, we set ${\mathcal G}(I)_+= \bigoplus_{n \geq 1}I^{n}/I^{n + 1}$. 

Here it is worth mentioning for completeness that, under a suitable set of hypotheses, there is a close connection between the Cohen-Macaulayness of ${\mathcal G}(I, M)$ and the property ${\rm r}(I, M)\leq 1$. We refer to \cite[Theorem 3.8]{Pol-X}.

The following fact generalizes \cite[Lemma 4.8]{Oo} (also \cite[Proposition 4.1]{J-U}).

\begin{Lemma}$($\cite[Corollary 3]{Zamani}$)$\label{reg=reg}
Let $R$ be a  ring, $I$ an ideal of $R$ and $M$ a
finite $R$-module. Then $\mathrm{reg}\,{\mathcal R}(I, M) = \mathrm{reg}\,{\mathcal G}(I, M) $.
\end{Lemma}

\begin{Remark}\label{reg<r}\rm
     In the context of Proposition \ref{proptrung}, a consequence is the following inclusion of submodules of $I^nM$, $$[(x_{1}, \ldots, x_{i - 1})I^{n}M :_{M} x_{i}] \, \cap \, I^{n}M \, \subseteq  \, (x_{1}, \ldots, x_{i - 1})I^{n - 1}M$$ for all $n \geq \ell + 1$ and $i = 1, \ldots, s$. But this is then an equality, as
     $x_i[(x_{1}, \ldots, x_{i - 1})I^{n - 1}M]\subseteq (x_{1}, \ldots, x_{i - 1})I^{n}M$. It follows that (1) holds for all $n \geq \ell + 1$ and therefore, by Lemma \ref{aistheleast} and Lemma \ref{propleastnum}, we have $\mathrm{reg}\,{\mathcal R}(I, M) \leq \ell$. On the other hand, by Lemma \ref{reg=reg}, there is a general equality $\mathrm{reg}\,{\mathcal R}(I, M) = \mathrm{reg}\,{\mathcal G}(I, M)$. We conclude that $$\mathrm{reg}\,{\mathcal R}(I, M) \, = \, \mathrm{reg}\,{\mathcal G}(I, M) \, \leq \, \ell,$$ which thus generalizes \cite[Proposition 4.7(iv)]{Trung}.
\end{Remark}

\begin{Remark}\rm \label{cotareg}
     Let $k \geq 1$ be a positive integer. Making $\ell = {\rm r}_{J}(I,M) + k - 1$ in Proposition \ref{proptrung}, we retrieve the Noetherian part of \cite[Proposition 5.2]{Giral-Planas-Vilanova}. The inequality ${\rm r}_{J}(I,M) \leq \mathrm{reg}\,{\mathcal R}(I, M)$ follows easily from Lemma \ref{propleastnum}, whereas the bound $$\mathrm{reg}\,{\mathcal R}(I, M) \, \leq \, {\rm r}_{J}(I, M) + k - 1$$ is a consequence of Remark \ref{reg<r}. In particular, the case $k=1$ yields $\mathrm{reg}\,{\mathcal R}(I, M) =  {\rm r}_{J}(I,M)$.
\end{Remark}

Next, we record a result which provides another set of conditions under which the equality $\mathrm{reg}\,{\mathcal R}(I, M) =  {\rm r}_{J}(I,M)$ holds. It will be a key ingredient in the proof of Theorem \ref{Mafigene}. For completeness, we recall that the condition on $x_1, \ldots, x_s$ present in part (a) below means that no element of the sequence lies in the ideal generated by the others and in addition there are equalities $0 :_Mx_{1}x_j = 0 :_Mx_j$ for $j=1, \ldots, s$ and $(x_1, \ldots, x_i)M :_Mx_{i+1}x_j = (x_1, \ldots, x_i)M :_Mx_j$ for $i=1, \ldots, s-1$ and $j=i+1, \ldots, s$. We observe that parts (a) and (b) are satisfied if the elements $x_1, \ldots, x_s$ form an $M$-sequence.

\begin{Lemma}$($\cite[Theorem 5.3]{Giral-Planas-Vilanova}$)$\label{GPV}
Let $R$ be a  ring, $I$ an ideal of $R$ and $M$ a
finite $R$-module. Let $J = (x_{1}, \ldots, x_{s})$ be a reduction of $I$ relative to $M$, and ${\rm r}_{J}(I, M) = {\rm r}$. Suppose the following conditions hold:
\begin{itemize}
    \item[(a)] $x_{1}, \ldots, x_{s}$ is a $d$-sequence relative to $M$;
    \item[(b)] $x_{1}, \ldots, x_{s - 1}$ is an $M$-sequence;
    \item[(c)] $(x_{1}, \ldots, x_{i})M \cap I^{{\rm r} + 1}M = (x_{1}, \ldots, x_{i})I^{{\rm r}}M$, \, $i = 1, \ldots, s - 1.$
\end{itemize}
Then, $\mathrm{reg}\,{\mathcal R}(I, M) =  {\rm r}_{J}(I,M)$.
\end{Lemma}

\subsection{On a question of Rossi, Trung, and Trung}\label{RTT} In this part we focus on the interesting question raised in \cite[Remark 2.5]{Rossi-Dinh-Trung} as to whether, for a ring $R$ and an ideal $I$ having a minimal reduction $J$, the formula $${\rm reg}\,\mathcal{R}(I) \, = \, \mathrm{min}\,\{n \geq {\rm r}_{J}(I) \, \mid \, I^{n + 1} : I = I^{n}\}$$ holds under the hypotheses of Lemma \ref{rossidinhtrung} below. This lemma is a crucial ingredient in the proof of Theorem \ref{rossiquest}, which will lead us to partial answers to the question as will be explained in Remark \ref{partial} and recorded in Corollary \ref{our-answer}. 

\begin{Lemma}$($\cite[Theorem 2.4]{Rossi-Dinh-Trung}$)$\label{rossidinhtrung}
Let $(R, \M)$ be a two-dimensional Buchsbaum local ring with $\mathrm{depth}\,R > 0$. Let $I$ be an $\M$-primary ideal which is not a parameter ideal, and let $J$ be a minimal reduction of $I$. Then
$${\rm reg}\,\mathcal{R}(I) \, = \, \mathrm{max}\,\{{\rm r}_{J}(I),\,s^{*}(I)\} = \mathrm{min}\,\{n \geq {\rm r}_{J}(I) \, \mid \, I^{n} = \widetilde{I^n}\}.$$
\end{Lemma}

\begin{Theorem}\label{rossiquest}
For $R$, $I$, and $J$ exactly as in Lemma \ref{rossidinhtrung}, there is an equality
$${\rm reg}\,\mathcal{R}(I) \, = \, \mathrm{min}\,\{n \geq {\rm r}_{J}(I) \, \mid \, I^{m + 1} : I = I^{m} \ \ \mathit{for \ all} \ \ m \geq n\}.$$
\end{Theorem}
\begin{proof} By Lemma \ref{rossidinhtrung}, ${\rm reg}\,\mathcal{R}(I) = \mathrm{max}\{{\rm r}_{J}(I), s^{*}(I)\}$. Assume first that $s^{*}(I) \leq {\rm r}_{J}(I)$. Note that, being an $\M$-primary ideal, $I$ contains a regular element since $\mathrm{depth}\,R > 0$. Thus, Proposition \ref{regelem}(a) yields $I^{m + 1} : I = \widetilde{I^{m + 1}} : I = \widetilde{I^{m}} = I^{m}$ for all $m \geq {\rm r}_{J}(I)$, and therefore  $${\rm reg}\,\mathcal{R}(I) \, = \, {\rm r}_{J}(I) \, = \, \mathrm{min}\,\{n \geq {\rm r}_{J}(I) \, \mid \, I^{m + 1} : I = I^{m} \ \mathrm{for \ all} \ m \geq n\}.$$ Now suppose ${\rm r}_{J}(I) < s^{*}(I)$. Thus ${\rm reg}\,\mathcal{R}(I) = s^{*}(I)$. On the other hand, Proposition \ref{s*conjec} gives
$s^{*}(I)  = \mathrm{min}\{m \geq 1 \mid I^{n + 1} : I = I^{n} \ \mathrm{for \ all} \ n \geq m \}$. Hence we can write ${\rm reg}\,\mathcal{R}(I) = \mathrm{min}\{m \geq {\rm r}_{J}(I) \mid I^{n + 1} : I = I^{n} \ \mathrm{for \ all} \ n \geq m\}$, as needed. \qed
\end{proof}

\begin{Remark}\label{partial}\rm We point out that if $s^{*}(I) \leq {\rm r}_{J}(I) + 1$ then our theorem (with the crucial description of $s^*(I)$ given by Proposition \ref{s*conjec}) settles affirmatively the problem of Rossi-Trung-Trung. This is easily seen to be true in the situation $s^{*}(I) \leq {\rm r}_{J}(I)$. Now suppose $s^{*}(I) = {\rm r}_{J}(I) + 1$, and write ${\rm r}_{J}(I)={\rm r}$. Then necessarily  $$I^{{\rm r} + 1} : I \, \neq \, I^{{\rm r}}.$$ It thus follows that Theorem \ref{rossiquest} solves the problem once again. Another situation where our result answers affirmatively the question is when $R$ is Cohen-Macaulay and $\widetilde{I^{\rm r}}=I^{\rm r}$; this holds because, by \cite[Remark 2.5]{Mafi}, we must have $s^{*}(I) \leq {\rm r}_{J}(I)$ in this case, and then we are done by a previous comment. Below we record such facts as a corollary.
\end{Remark}

\begin{Corollary}\label{our-answer}
Let $R$, $I$, and $J$ be exactly as in Lemma \ref{rossidinhtrung}. Suppose in addition any one of the following conditions: 
\begin{itemize}
    \item[(a)] $s^{*}(I) \leq {\rm r}_{J}(I) + 1$;
    \item[(b)] $R$ is Cohen-Macaulay and $\widetilde{I^{\rm r}}=I^{\rm r}$, where ${\rm r}={\rm r}_{J}(I)$.
    \end{itemize}
Then the Rossi-Trung-Trung question has an affirmative answer.

\end{Corollary}

As far as we know, under the hypotheses of Lemma \ref{rossidinhtrung} there is no example satisfying $s^*(I)>{\rm r}_{J}(I) + 1$. It is thus natural to raise the following question, to which an affirmative answer would imply the definitive solution of the Rossi-Trung-Trung problem.
 
\begin{Question}\rm Under the hypotheses of Lemma \ref{rossidinhtrung}, is it true that $s^{*}(I) \leq {\rm r}_{J}(I) + 1$?
 \end{Question}

Regarding the setting relative to a module, a question is in order.

\begin{Question}\rm Let $R$ be exactly as in Lemma \ref{rossidinhtrung}. Let $I$ be an $\M$-primary ideal and let $J$ be a minimal reduction of $I$ relative to a Buchsbaum $R$-module $M$, satisfying ${\rm r}_J(I, M)\geq 1$. We raise the question of whether any of the following equalities
\begin{itemize}
    \item[(a)] ${\rm reg}\,\mathcal{R}(I, M) = \mathrm{max}\,\{{\rm r}_{J}(I, M),\,s^{*}(I, M)\}$, 
    \item[(b)] ${\rm reg}\,\mathcal{R}(I, M) = \mathrm{min}\,\{n \geq {\rm r}_{J}(I, M) \mid I^{m + 1}M :_M I = I^{m}M \ \mathrm{for \ all} \ m \geq n\}$
    \end{itemize} is true. Notice that the assertions (a) and (b) are in fact equivalent; the implication (a)$\Rightarrow$(b) can be easily seen by using Theorem \ref{Mafigene} (to be proven in the next section) and Proposition \ref{s*conjec}, whereas (b)$\Rightarrow$(a) follows easily with the aid of Proposition \ref{s*conjec}. We might even conjecture that such equalities hold in the particular case where $R$ and $M$ are Cohen-Macaulay.
\end{Question}

\section{A generalization of a result of Mafi}\label{gen-of-Mafi}

Our main goal in this section is to establish Theorem \ref{Mafigene}, which concerns the interplay between the numbers ${\rm r}_{J}(I, M)$ and  $\mathrm{reg}\,{\mathcal R}(I, M)$ in a suitable setting. As we shall explain, our theorem generalizes a result due to Mafi from \cite{Mafi} (see Corollary \ref{Mafigene2}) and gives us a number of other applications, to be described in Section \ref{app} and also in Section \ref{Ulrich}. As a matter of notation, it is standard to write ${\rm grade}(I, M)$ for the maximal length of an $M$-sequence contained in the ideal $I$ of the local ring $(R, \M)$. If $M=R$, the notation is simplified to ${\rm grade}\,I$. Note that ${\rm grade}(I, M)$ is just the $I$-${\rm depth}$ of $M$; in particular, ${\rm grade}(\M, M)={\rm depth}\,M$. 

\smallskip

First we recall two general lemmas.

\begin{Lemma}$($\cite[Lemma 1.2]{Rossi-Valla}$)$\label{LemmaRV0}
Let $R$ be a local ring, $M$ a finite $R$-module of positive dimension, and $I$ a proper ideal. Let $x_1, \ldots, x_s$ be an $M$-superficial sequence of $I$. Then $x_1, \ldots, x_s$ is an $M$-sequence if and only if ${\rm grade}(I, M)\geq s$.
\end{Lemma}

\begin{Lemma}$($\cite[p.\,12]{Rossi-Valla}$)$\label{LemmaRV} Let $(R, \M)$ be a local ring with infinite residue field, $M$ a finite $R$-module of positive dimension, $I$ an $\M$-primary ideal, and $J$ a minimal reduction of $I$ relative to $M$. Then $J$ can be generated by an $M$-superficial sequence of $I$, which is also a system of parameters of $M$. 
\end{Lemma}

Our theorem (which also holds in an appropriate  graded setting) is as follows.

\begin{Theorem}\label{Mafigene}
	Let $(R, \M)$ be a local ring with infinite residue field, $M$ a Cohen-Macaulay $R$-module of dimension $s\geq 1$, $I$ an $\M$-primary ideal, and $J = (x_1, \ldots, x_s)$ a minimal reduction of $I$ relative to $M$. Set ${\rm r} = {\rm r}_{J}(I, M)$ and $M_j = M/(x_{1}, \ldots, x_{j - 1})M$ with $M_1=M$. Assume that either ${\rm r} = 0$, or ${\rm r} \geq 1$ and $\widetilde{I^{\rm r}_{M_{j}}}  =  I^{\rm r}M_{j}$ for $j = 1, \ldots, s-1$ {\rm (}if $s\geq 2${\rm )}.  Then
	$$\mathrm{reg}\,{\mathcal R}(I, M) \, = \, \mathrm{reg}\,{\mathcal G}(I, M) \, = \, {\rm r}_{J}(I, M).$$ In particular, ${\rm r}(I, M)$ is independent.
\end{Theorem}
\begin{proof} First, by Lemma \ref{reg=reg}, the equality $\mathrm{reg}\,{\mathcal R}(I, M) = \mathrm{reg}\,{\mathcal G}(I, M)$ holds. Thus our objective is to prove that $\mathrm{reg}\,{\mathcal R}(I, M) = {\rm r}$. 

Because $M$ is Cohen-Macaulay and $I$ is $\M$-primary, we have $s  =  \mathrm{depth}\,M =  \mathrm{grade}(I, M)$
and then $x_{1}, \ldots, x_{s}$ must be in fact an $M$-sequence by Lemma \ref{LemmaRV0}. As a consequence, conditions (a) and (b) of Lemma \ref{GPV} are satisfied. We shall prove that \begin{equation}\label{main}(x_{1}, \ldots, x_{i})M \, \cap \, I^{{\rm r} + 1}M \, = \, (x_{1}, \ldots, x_{i})I^{{\rm r}}M, \ \ \ \ \ i = 1, \ldots, s - 1 \end{equation} and then the desired formula $\mathrm{reg}\,{\mathcal R}(I, M) = {\rm r}$ will follow immediately by Lemma \ref{GPV}. Notice that the condition is vacuous if $s=1$ (both sides of (\ref{main}) are regarded as zero), so we can suppose $s\geq 2$ from now on.
		
The case ${\rm r} = 0$ is now trivial because $(x_{1}, \ldots, x_{i})M \subseteq IM$, or what amounts to the same, $(x_{1}, \ldots, x_{i})M \cap IM = (x_{1}, \ldots, x_{i})M$, for all $i = 1, \ldots, s - 1$.

%$$ (I^{n})_{i} \, = \, \frac{I^n + (x_{1}, \ldots, x_{i - %1})}{(x_{1}, \ldots, x_{i - 1})},$$

So let us assume that ${\rm r}\geq 1$ and $\widetilde{I^{\rm r}_{M_{j}}}  =  I^{\rm r}M_{j}$ for $j = 1, \ldots, s-1$. We  proceed by induction on $i$. First, pick an element $f  \in  (x_{1})M  \cap  I^{{\rm r} + 1}M$. Then there exists $m \in M$ such that $f = x_{1}m \in I^{{\rm r} + 1}M$. By Proposition \ref{regelem}(b) and the hypothesis that $\widetilde{I^{{\rm r}}_{M}} \, = \, I^{{\rm r}}M$, we have $$m \, \in \, I^{{\rm r} + 1}M :_{M} x_{1} \, \subseteq \, \widetilde{I^{{\rm r} + 1}_{M}} :_{M} x_{1} \, = \, \widetilde{I^{{\rm r}}_{M}} \, = \, I^{{\rm r}}M,$$ so that $f = x_{1}m \in (x_{1})I^{{\rm r}}M$, which shows $(x_{1})M  \cap  I^{{\rm r} + 1}M  =  (x_{1})I^{{\rm r}}M$. Now let $i\geq 2$ and suppose \begin{equation}\label{suppose}
	(x_{1}, \ldots, x_{i - 1})M \, \cap \, I^{{\rm r} + 1}M \, = \, (x_{1}, \ldots, x_{i - 1})I^{{\rm r}}M.
\end{equation} We claim that
\begin{equation}\label{claimed}
	[I^{{\rm r} + 1}M + (x_{1}, \ldots, x_{i - 1})M] \, \cap \, (x_{1}, \ldots, x_{i})M \, = \, (x_{1}, \ldots, x_{i - 1})M  +  x_{i}I^{{\rm r}}M.
\end{equation} The inclusion 
$(x_{1}, \ldots, x_{i - 1})M  +  x_{i}I^{{\rm r}}M\subseteq [I^{{\rm r} + 1}M + (x_{1}, \ldots, x_{i - 1})M] \, \cap \, (x_{1}, \ldots, x_{i})M$ is clear. Now pick $$g \, \in \, [I^{{\rm r} + 1}M + (x_{1}, \ldots, x_{i - 1})M] \cap (x_{1}, \ldots, x_{i})M.$$ Then, there exist $h \in I^{{\rm r} + 1}M$ and $m_{l}, n_{k} \in M$, with $l = 1, \ldots, i - 1$ and $k = 1, \ldots, i$, such that
$$g \, = \, h + \sum_{l = 1}^{i - 1}\,x_{l}m_{l} \, = \, \sum_{k = 1}^{i}\,x_{k}n_{k}.$$ Hence $h - x_{i}n_{i} \in (x_{1}, \ldots, x_{i - 1})M$. Now, denote 
$ (I^{n})_{i} = (I^n + (x_{1}, \ldots, x_{i - 1}))/(x_{1}, \ldots, x_{i - 1})$ for any given $n\geq 0$, which agrees with $(I_i)^n$. Since $M_i  =  M/(x_{1}, \ldots, x_{i - 1})M$, we have  \begin{equation}\label{equal0}
 \widetilde{I^{n}_{M_{i}}}  \, = \, \widetilde{(I_i)^{n}_{M_{i}}} 
\end{equation} 
for all $n \geq 1$ by  \cite[Observation 2.3]{Puthenpurakal1}. Modulo $(x_{1}, \ldots, x_{i - 1})M$, we can write
\begin{equation}\label{equal1}
    \overline{x_{i}n_{i}} \, = \, \overline{h} \, \in \, I^{{\rm r} + 1}M_i.
\end{equation}
Note that $\mathrm{grade}(I_i, M_i) > 0$ as the image $\overline{x_i} \in I_i$ of $x_{i}$ is $M_i$-regular. Moreover, $\overline{x_i}$ is $M_i$-superficial. Applying (\ref{equal0}), (\ref{equal1}), Proposition \ref{regelem}(b) and our hypotheses, we get \begin{equation}\label{all} \overline{n_{i}} \, \in \, I^{{\rm r} + 1}{M_i} :_{M_i} \overline{x_{i}} \, \subseteq \, \widetilde{I^{{\rm r} + 1}_{M_i}} :_{M_i} \overline{x_{i}} \, = \, \widetilde{(I_i)^{{\rm r} + 1}_{M_i}} :_{M_i} \overline{x_{i}} \, = \, 
 \widetilde{(I_i)^{{\rm r}}_{M_i}} \, = \, \widetilde{I^{{\rm r}}_{M_i}} \, = \, I^{{\rm r}}M_{i},\end{equation} where in particular we are using that $\widetilde{I^{{\rm r}}_{M_s}}  =  I^{{\rm r}}M_{s}$, which we claim to hold and will confirm in the last part of the proof. So, (\ref{all}) yields $n_{i} \in I^{{\rm r}}M + (x_{1}, \ldots, x_{i - 1})M$,
and this implies $$g \, = \, \left(\displaystyle{\sum_{k = 1}^{i-1}\,x_{k}n_{k}}\right) + x_in_i \, \in \, (x_{1}, \ldots, x_{i - 1})M + x_{i}I^{{\rm r}}M,$$ which thus proves (\ref{claimed}). Now, using  (\ref{claimed}) and (\ref{suppose}) we obtain
	\begin{eqnarray}
		I^{{\rm r} + 1}M \cap (x_{1}, \ldots, x_{i})M &=& I^{{\rm r} + 1}M \cap (I^{{\rm r} + 1}M + (x_{1}, \ldots, x_{i - 1})M) \cap (x_{1}, \ldots, x_{i})M \nonumber \\ ~ &=& I^{{\rm r} + 1}M \cap [(x_{1}, \ldots, x_{i - 1})M + x_{i}I^{{\rm r}}M] \nonumber \\~ &=& [I^{{\rm r} + 1}M \cap (x_{1}, \ldots, x_{i - 1})M] + x_{i}I^{{\rm r}}M \nonumber \\
		~ &=& (x_{1}, \ldots, x_{i - 1})I^{{\rm r}}M + x_{i}I^{{\rm r}}M \nonumber \\ ~ &=& (x_{1}, \ldots, x_{i})I^{{\rm r}}M, \nonumber 
		\end{eqnarray} which gives (\ref{main}), as needed. 
		
		To finish the proof, it remains to verify $\widetilde{I^{{\rm r}}_{M_s}}  =  I^{{\rm r}}M_{s}$. First observe that, for all $n \geq {\rm r}$, $$I^{n + 1}M = JI^{n}M = (x_1, \ldots, x_s)I^{n}M = x^{s}I^{n}M + (x_1, \ldots, x_{s - 1})I^{n}M.$$
Adding $(x_1, \ldots, x_{s - 1})M$ to both  sides, we have $$\begin{array}{cccc}
	      ~  & x_sI^{n}M + (x_1, \ldots, x_{s - 1})M &=& I^{n + 1}M + (x_1, \ldots, x_{s - 1})M \\ \Rightarrow & x_sJI^{n - 1}M + (x_1, \ldots, x_{s - 1})M &=& I^{n + 1}M + (x_1, \ldots, x_{s - 1})M \\ \Rightarrow & x_s[ x_sI^{n - 1}M + (x_1, \ldots, x_{s - 1})I^{n - 1}M] + (x_1, \ldots, x_{s - 1})M &=& I^{n + 1}M + (x_1, \ldots, x_{s - 1})M  \\ \Rightarrow & x_s^{2}I^{n - 1}M + (x_1, \ldots, x_{s - 1})M &=& I^{n + 1}M + (x_1, \ldots, x_{s - 1})M. 
	\end{array}$$
	Proceeding with the same argument, we deduce $$x_s^{n - {\rm r} + 1}I^{\rm r}M + (x_1, \ldots, x_{s - 1})M = I^{n + 1}M + (x_1, \ldots, x_{s - 1})M.$$ Using this equality, the standard definitions and the fact that $x_s$ is $M/(x_1, \ldots, x_{s-1})M$-regular, we can take $k \gg 0$ so that
	$$\begin{array}{ccccccccc}
	    \widetilde{I^{\rm r}_{M_{s}}} & = &  \widetilde{(I_s)^{\rm r}_{M_s}} & = & \left(\frac{I}{(x_1, \ldots, x_{s - 1})}\right)^{{\rm r} + k}M_s :_{M_s}  \left(\frac{I}{(x_1, \ldots, x_{s - 1})}\right)^{k} \nonumber \\
	   ~ & = & \frac{I^{{\rm r} + k}M + (x_1, \ldots, x_{s - 1})M}{(x_1, \ldots, x_{s - 1})M} :_{M_s}  \frac{I^{k} + (x_1, \ldots, x_{s - 1})}{(x_1, \ldots, x_{s - 1})} & \subseteq & \frac{I^{{\rm r} + k}M + (x_1, \ldots, x_{s - 1})M}{(x_1, \ldots, x_{s - 1})M} :_{M_s} \overline{x_s^{k}} \nonumber \\ ~ & = & \frac{x^{k}_sI^{\rm r}M + (x_1, \ldots, x_{s - 1})M}{(x_1, \ldots, x_{s - 1})M} :_{M_s} \overline{x_s^{k}} & = & \frac{I^{\rm r}M + (x_1, \ldots, x_{s - 1})M}{(x_1, \ldots, x_{s - 1}))M} \nonumber \\ ~ & = & I^{\rm r}M_s.\end{array}$$ \qed
	\end{proof}

\medskip

It seems natural to raise the following question.

\begin{Question}\rm \label{questionmafi} In Theorem \ref{Mafigene} (assuming $s\geq 3$ and ${\rm r}\neq 0$), can we replace the set of hypotheses $\widetilde{I^{\rm r}_{M_{j}}}  =  I^{\rm r}M_{j}$, \,$j = 1, \ldots, s-1$, with the single condition $\widetilde{I^{\rm r}_M} = I^{\rm r}M$\,? 
%Here, ${\rm r}={\rm r}_{J}(I, M)\neq 0$.
 \end{Question}

The result below (the case $s=2$) is an immediate byproduct of Theorem \ref{Mafigene}.

\begin{Corollary} \label{corMafigen}
    	Let $(R, \M)$ be a local ring with infinite residue field, $M$ a Cohen-Macaulay $R$-module with ${\rm dim}\,M = 2$, $I$ an $\M$-primary ideal, and $J$ a minimal reduction of $I$ relative to $M$. Setting ${\rm r = r}_{J}(I, M)$, assume that either ${\rm r} = 0$, or ${\rm r} \geq 1$ and $\widetilde{I^{\rm r}_{M}}  = I^{\rm r}M$.  Then
	$$\mathrm{reg}\,{\mathcal R}(I, M) \, = \, \mathrm{reg}\,{\mathcal G}(I, M) \, = \, {\rm r}_{J}(I, M).$$ 
\end{Corollary}

We are now able to recover an interesting result of Mafi about zero-dimensional ideals in two-dimensional Cohen-Macaulay local rings. More precisely, by taking $M = R$ in  Corollary \ref{corMafigen} we readily derive the following fact.

\begin{Corollary}$($\cite[Proposition 2.6]{Mafi}$)$\label{Mafigene2} Let $(R, \M)$ be a two-dimensional Cohen-Macaulay local ring with infinite residue field. Let $I$ be an $\M$-primary ideal and $J$ a minimal reduction of $I$.  Setting ${\rm r = r}_{J}(I)$, assume that either ${\rm r} = 0$, or ${\rm r} \geq 1$ and $\widetilde{I^{\rm r}} = I^{\rm r}$. Then $$\mathrm{reg}\,{\mathcal R}(I) \, = \, \mathrm{reg}\,{\mathcal G}(I) \, = \, {\rm r}_{J}(I).$$ 
\end{Corollary}

%(see \cite[(1.2), p.\,594]{Heinzer-Lantz-Shah})

\begin{Example}\rm Consider the ideal $I  =  (x^6, x^4y^2,x^3y^3,y^6)$ of the formal power series ring $R = k[[x,y]]$, where $k$ is a field. By \cite[Example 3.2]{Huckaba}, we have  ${\rm r}(I) = 3$ and $\mathrm{grade}\,\mathcal{G}(I)_{+} > 0$. This latter fact is equivalent to all powers of $I$ being Ratliff-Rush closed 
(see \cite[Fact 9]{Heinzer-Johnson-Lantz-Shah}), i.e.,  $s^{*}(I) = 1$. In particular, $\widetilde{I^3} = I^3$, and therefore Corollary \ref{Mafigene2} gives $$\mathrm{reg}\,\mathcal{R}(I) \, = \, \mathrm{reg}\,\mathcal{G}(I) \, = \, 3.$$ Let us also mention that, in this case, it is easy to alternatively confirm any of these equalities, say $\mathrm{reg}\,\mathcal{R}(I)=3$, by first regarding (for simplicity) $I$ as an ideal of the polynomial ring $k[x, y]$ and then noting that the Rees algebra of $I$ can be presented as $\mathcal{R}(I)  =  S/\mathcal{J}  =  k[x, y, Z, W, T, U]/\mathcal{J}$, where now $x, y$ are seen in degree 0 and $Z, W, T, U$ are indeterminates of degree 1 over $k[x, y]$. Explicitly, $\mathcal{J}$ is the homogeneous $S$-ideal $$\mathcal{J} \, = \, (T^2-ZU,\, yW-xT,\, W^3-Z^2U,\, y^2Z-x^2W,\, xW^2-yZT,\, y^3T-x^3U).$$ A computation gives the shifts of the graded $S$-free modules along a (length $3$) resolution of $\mathcal{J}$, and we can finally use a standard device (e.g., \cite[Exercise 15.3.7(iv)]{B-S}) to get ${\rm reg}\,\mathcal{J}=4$, hence 
$\mathrm{reg}\,\mathcal{R}(I)=\mathrm{reg}\,S/\mathcal{J}={\rm reg}\,\mathcal{J}-1=3$, as desired.

%of the form $$0\rightarrow S^2(-7)\rightarrow S(-5)\oplus %S^7(-6)\rightarrow S^5(-4)\oplus S^6(-5)\rightarrow %S^2(-2)\oplus S^3(-3)\oplus S(-4).$$ Now, using %\cite[Exercise 15.3.7(iv)]{B-S}, we obtain ${\rm %reg}\,\mathcal{J}={\rm max}\{4-0,\, 5-1,\, 6-2,\, 7-3\}=4$, %so that $\mathrm{reg}\,\mathcal{R}(I)=\mathrm{reg}\,S/\math%cal{J}={\rm reg}\,\mathcal{J}-1=3$, as desired.

\end{Example}

To conclude the section, we illustrate that Corollary \ref{Mafigene2} is no longer valid if we remove the hypothesis involving the Ratliff-Rush closure.

\begin{Example}\rm
	 Consider the monomial ideal $$I  \, = \, (x^{157},\, x^{35}y^{122},\, x^{98}y^{59},\, y^{157}) \, \subset \, k[x, y],$$ where $k[x,y]$ is a (standard graded) polynomial ring over a field $k$. By \cite[Proposition 3.3 and Remark 3.4]{DTrung}, the ideal $J = (x^{157},y^{157})$ is a minimal reduction of $I$ satisfying $${\rm r}_{J}(I) \, = \, 20 \, < \, 21 \, = \, \mathrm{reg}\,{\mathcal R}(I) \, = \, s^{*}(I),$$
	where the last equality follows by Lemma \ref{rossidinhtrung}. Thus, ${\rm r}_{J}(I) = s^{*}(I) - 1\neq \mathrm{reg}\,{\mathcal R}(I)$. Finally notice that, as $s^{*}(I)={\rm r}_J(I)+1=20+1$, we have $\widetilde{I^{20}} \neq I^{20}$.

\end{Example}

\section{First applications}\label{app}

We now describe some applications of the results obtained in the preceding section. With the exception of Subsection \ref{lintype}, we shall focus on the two-dimensional case.

\subsection{Hyperplane sections of Rees modules} The first application is the corollary below, which can be of particular interest for inductive arguments in dealing with Castelnuovo-Mumford regularity of Rees modules.

\begin{Corollary}
	Let $(R, \M)$ be a local ring with infinite residue field, $I$ an $\M$-primary ideal, and $M$ a Cohen-Macaulay $R$-module with $\mathrm{dim}\,M = 2$. Suppose $\widetilde{I^{\rm r}_{M}} = I^{\rm r}M$ with ${\rm r} = {\rm r}(I, M)$. Let $x \in \M \setminus I$ be an $M$-regular element whose initial form in $\mathcal{G}(I)$ is $\mathcal{G}(I, M)$-regular. If $$\widetilde{\left(\frac{I + (x)}{(x)}\right)}_{\frac{M}{xM}} = \frac{IM + xM}{xM}$$ then\,
	$\mathrm{reg}\,\mathcal{R}(I, M)/x\mathcal{R}(I, M) = \mathrm{reg}\,\mathcal{R}(I, M).$
\end{Corollary}
\begin{proof} It follows from \cite[Lemma 2.3]{Zamani1} that $$\mathrm{reg}\,\mathcal{R}(I, M)/x\mathcal{R}(I, M) \, = \, \mathrm{reg}\,\mathcal{R}((I + (x))/(x),\,M/xM).$$
Since $M/xM$ is a Cohen-Macaulay $R/(x)$-module of positive dimension (equal to $1$) and the ideal $(I + (x))/(x)$ is $\M/(x)$-primary, we can apply Corollary \ref{corMafigen} so as to obtain $$\mathrm{reg}\,\mathcal{R}((I + (x))/(x),\,M/xM) \, = \, {\rm r}((I + (x))/(x),\,M/xM).$$ On the other hand, \cite[Lemma 2.2]{Zamani1} yields ${\rm r}((I + (x))/(x),\,M/xM)  =  {\rm r}(I,\,M)$. Now, again by Corollary \ref{corMafigen}, ${\rm r}(I, M) = \mathrm{reg}\,\mathcal{R}(I, M)$. Therefore, the asserted equality is true. \qed
\end{proof}

\medskip

Taking $M=R$, we immediately get the following consequence.

\begin{Corollary}
	Let $(R, \M)$ be a two-dimensional Cohen-Macaulay local ring with infinite residue field, and $I$ an $\M$-primary ideal. Suppose $\widetilde{I^{\rm r}} = I^{\rm r}$ with ${\rm r} = {\rm r}(I)$. Let $x \in \M \setminus I$ be a regular element whose initial form in $\mathcal{G}(I)$ is regular. If the $R/(x)$-ideal  $(I + (x))/(x)$ is Ratliff-Rush closed, then\,
	$\mathrm{reg}\,\mathcal{R}(I)/x\mathcal{R}(I) = \mathrm{reg}\,\mathcal{R}(I).$
\end{Corollary}

\subsection{The role of postulation numbers} This subsection (which focuses on the classical case $M=R$) investigates connections of postulation numbers with the Castelnuovo-Mumford regularity of blowup algebras and reduction numbers. First recall that, if $(R, \M)$ is a local ring and $I$ is an $\M$-primary ideal of $R$, then the corresponding {\it Hilbert-Samuel function} is given by $$H_{I}(n) \, = \, \displaystyle{\lambda\left(R / I^{n}\right)}$$ for any integer $n \geq 1$, and $H_{I}(n) = 0$ if $n \leq 0$. The symbol $\lambda(-)$ denotes length of $R$-modules. It is well-known that $H_I(n)$ coincides, for all sufficiently large integers $n$, with a polynomial $P_I(n)$ -- the {\it Hilbert-Samuel polynomial of $I$}.

\begin{Definition}\rm The \textit{postulation number} of $I$ is the integer
$$\rho(I) \, = \, \mathrm{sup}\{n \in \mathbb{Z} \mid H_{I}(n) \neq P_{I}(n)\}.$$
\end{Definition}

Here it is worth recalling that the functions $H_I(n)$ and $P_I(n)$ are defined for all integers $n$, so $\rho(I)$ can be -- and often is -- negative (as emphasized in \cite[Introduction, p.\,335]{Marley2}).

In the application below we furnish a characterization of $\mathrm{reg}\,\mathcal{R}(I)$ (which by \cite[Lemma 4.8]{Oo} agrees with $\mathrm{reg}\,\mathcal{G}(I)$), in terms, in particular, of $\rho(I)$. Notice that our statement makes no explicit use of the concept of postulation number $p(\mathcal{G}(I))$ for the ring $\mathcal{G}(I)$ (see \cite[Remark 1.1]{Strunk}), which we only use in the proof.

\begin{Corollary}\label{strunk}
    Let $(R,\M)$ be a two-dimensional Cohen–Macaulay local ring with infinite residue field, and let $I$ be an $\M$-primary ideal with a minimal reduction $J$. If $\mathrm{grade}\,\mathcal{G}(I)_{+} = 0$, then
    $$\mathrm{reg}\,\mathcal{R}(I) = {\rm max}\{{\rm r}_J(I),\,\rho(I) + 1\}.$$
\end{Corollary}
\begin{proof} We have ${\rm dim}\,\mathcal{G}(I)=2$ and  $H^j_{\mathcal{G}(I)_+}(\mathcal{G}(I))=0$ for all $j>2$, so that $a_j(\mathcal{G}(I))=-\infty$ whenever $j>2$. Moreover, $\mathrm{grade}\,\mathcal{G}(I)_{+} = 0$ and then $a_{0}(\mathcal{G}(I)) < a_{1}(\mathcal{G}(I))$ by virtue of \cite[Theorem 2.1(a)]{Marley2}.
    Therefore, $$\mathrm{reg}\,\mathcal{G}(I) \, = \, {\rm max}\{a_1(\mathcal{G}(I))+1,\, a_2(\mathcal{G}(I))+2\}.$$ Now let us consider the subcase where $a_{1}(\mathcal{G}(I)) \leq a_{2}(\mathcal{G}(I))$, so that $\mathrm{reg}\,\mathcal{G}(I)  =  a_{2}(\mathcal{G}(I)) + 2$. On the other hand, \cite[Lemma 1.2]{Marley2} yields $$a_{2}(\mathcal{G}(I)) + 2  \, \leq \, {\rm r}_{J}(I) \, \leq \, \mathrm{reg}\,\mathcal{G}(I),$$ and it follows that $\mathrm{reg}\,\mathcal{G}(I)= {\rm r}_{J}(I)$. 
    Finally, if $a_{2}(\mathcal{G}(I)) < a_{1}(\mathcal{G}(I))$, then $a_{2}(\mathcal{G}(I)) +2\leq a_{1}(\mathcal{G}(I))+1$ and hence $\mathrm{reg}\,\mathcal{G}(I)  =  a_{1}(\mathcal{G}(I)) + 1$. In addition, \cite[Proof of Theorem 3.10]{Strunk} gives $\rho(I)=p(\mathcal{G}(I))$, and on the other hand, applying \cite[Corollary 2.3(2)]{Brodmann-Linh} we can write  $p(\mathcal{G}(I))=a_{1}(\mathcal{G}(I))$ since $a_{1}(\mathcal{G}(I))$ is strictly bigger than both $a_{0}(\mathcal{G}(I))$ and $a_{2}(\mathcal{G}(I))$. Thus, $\mathrm{reg}\,\mathcal{G}(I) = \rho(I) + 1$. \qed
\end{proof}

\medskip

Next we provide a different proof (in fact an improvement) of \cite[Proposition 3.7]{Hoa}. Notice that Hoa's $c(I)$ is just $s^*(I)-1$.

\begin{Corollary}
    	Let $(R, \M)$ be a two-dimensional Cohen-Macaulay local ring with infinite residue field, and let $I$ be an $\M$-primary ideal. If $\rho(I)  \neq  s^{*}(I) - 1$ then ${\rm r}_{J}(I)=\mathrm{reg}\,\mathcal{R}(I)$ for any minimal reduction $J$ of $I$. In particular,
    	${\rm r}(I)$ is independent.
\end{Corollary}
\begin{proof} Set ${\rm r}={\rm r}(I)$. By virtue of Corollary \ref{Mafigene2}, we can assume that ${\rm r}(I)\geq 1$ (in particular, $I$ cannot be a parameter ideal; see Subsection \ref{lintype} below). Let us consider first the case $\mathrm{grade}\,\mathcal{G}(I)_{+} >0$. According to \cite[Fact 9]{Heinzer-Johnson-Lantz-Shah}, all powers of $I$ must be Ratliff-Rush closed, and hence $s^*(I)=1$. Note ${\rm r}(I)\geq 1$ implies ${\rm r}_J(I)\geq 1=s^*(I)$ for every minimal reduction $J$ of $I$. By Lemma \ref{rossidinhtrung}, the asserted equality follows.

So we can suppose $\mathrm{grade}\,\mathcal{G}(I)_{+} = 0$. As verified in the proof of Corollary \ref{strunk}, if $a_{1}(\mathcal{G}(I)) \leq a_{2}(\mathcal{G}(I))$ then  for any minimal reduction $J$ of $I$ we can write
${\rm r}_{J}(I)=\mathrm{reg}\,\mathcal{G}(I)$, which as we know coincides with $\mathrm{reg}\,\mathcal{R}(I)$. Moreover, we have seen that if $a_{2}(\mathcal{G}(I)) < a_{1}(\mathcal{G}(I))$ then $\mathrm{reg}\,\mathcal{G}(I) = \rho(I) + 1$, so that  $\mathrm{reg}\,\mathcal{G}(I)\neq s^*(I)$, which is tantamount to saying that $\mathrm{reg}\,\mathcal{R}(I)  \neq  s^*(I)$. 
Now Lemma \ref{rossidinhtrung} yields $\mathrm{reg}\,\mathcal{R}(I) =  {\rm r}_{J}(I)$ whenever $J$ is a minimal reduction of $I$, as needed. \qed
\end{proof}

%\medskip

%From the proof above we readily derive the %following corollary. Note that, if the %$\mathfrak{m}$-primary ideal $I$ is a parameter %ideal  -- in which case the reduction number is %known to be zero -- then we have %$\mathrm{reg}\,\mathcal{R}(I)=0$ (see, e.g., %the case $r=0$ of \cite[Theorem 1.1]{Trung}).

%\begin{Corollary}
%    	Let $(R, \M)$ be a two-dimensional %Cohen-Macaulay local ring with infinite residue %field, and let $I$ be an $\M$-primary ideal. If %$\rho(I)  \neq s^{*}(I) - 1$ and %$\mathrm{grade}\,\mathcal{G}(I)_{+} = 0$, then %${\rm r}_{J}(I)  =  %\mathrm{reg}\,\mathcal{R}(I)$ for any minimal %reduction $J$ of $I$.
%    \end{Corollary}

\subsection{Ideals of linear type}\label{lintype}

For an ideal $I$ of a ring $R$, there is a canonical homomorphism $\pi \colon \mathcal{S}(I)\rightarrow \mathcal{R}(I)$ from the symmetric algebra $\mathcal{S}(I)$ of $I$ onto its Rees algebra $\mathcal{R}(I)$. The ideal $I$ is said to be of {\it linear type} if $\pi$ is an isomorphism. To see what this means a bit more concretely, we can make use of some (any) $R$-free presentation $$R^m\stackrel{\varphi}{\longrightarrow} R^{\nu}\longrightarrow I\longrightarrow 0.$$ Letting $S=R[T_1, \ldots, T_{\nu}]$ be a standard graded polynomial ring in variables $T_1, \ldots, T_{\nu}$ over $R=S_0$, we can identify $\mathcal{S}(I)=S/\mathcal{L}$, where $\mathcal{L}$ is the ideal generated by the $m$ linear forms in the $T_i$'s given by the entries of the product $(T_1 \cdots T_{\nu})\cdot \varphi$, where $\varphi$ is regarded as a $\nu \times m$ matrix taken with respect to the canonical bases of $R^{\nu}$ and $R^m$. We can also write $\mathcal{R}(I)=S/\mathcal{J}$, for a certain ideal $\mathcal{J}$ containing $\mathcal{L}$. Precisely, expressing $\mathcal{R}(I)=R[It]$ (where $t$ is an indeterminate over $R$), then $\mathcal{J}$ is the kernel of the natural epimorphism $S\rightarrow \mathcal{R}(I)$. Now, the above-mentioned map $\pi$ can be simply identified with the surjection $S/\mathcal{L}\rightarrow S/\mathcal{J}$. It follows that $I$ is of linear type if and only if $\mathcal{J}=\mathcal{L}$. For instance, if $I$ is generated by a regular sequence then $I$ is of linear type (see \cite[5.5]{Swanson-Huneke}).

%This is an important, well-studied notion %in commutative algebra. 

Now assume $(R, \M)$ is either local with infinite residue field or a standard graded algebra over an infinite field. The {\it analytic spread} of $I$, which we denote $s(I)$, is defined as the Krull dimension of the special fiber ring $\mathcal{R}(I)/\M \mathcal{R}(I)$. It is a classical fact that ${\rm r}(I)=0$ if and only if $I$ can be generated by $s(I)$ elements (see \cite[Theorem 4(ii)]{NR}). So the property ${\rm r}(I)=0$ is easily seen to take place, for example, whenever $I$ is of linear type (or if $I$ is an $\M$-primary parameter ideal). 

While, in general, there exist examples of ideals with reduction number zero that are not of linear type, the natural question remains as to under what conditions the property ${\rm r}(I)=0$ forces $I$ to be of linear type. The corollary below is a quick application of Theorem \ref{Mafigene} (along with the property that the Castelnuovo-Mumford regularity controls degrees over graded polynomial rings) and contributes to this problem in a suitable setting. It is plausible to believe that the result is known to experts, but, as far as we know, it has not yet been recorded in the literature.

%Let $(R, \M)$ be a Cohen-Macaulay local %ring with infinite residue field and %positive dimension, and

\begin{Corollary}\label{app-lintype}
	Let $R$ be a standard graded polynomial ring over an infinite field, and let $I$ be a zero-dimensional ideal of $R$. If ${\rm r}(I) = 0$, then $I$ is of linear type.
	\end{Corollary}
\begin{proof} Applying the graded analogue of Theorem \ref{Mafigene} with $M=R$, we obtain $\mathrm{reg}\,{\mathcal R}(I)=0$. Using the above notations, we get $\mathrm{reg}\,S/{\mathcal J}=0$, so that $\mathrm{reg}\,{\mathcal J}=1$, which implies that the homogeneous $S$-ideal $\mathcal{J}$ admits no minimal generator of degree greater than 1. In other words, we must have  ${\mathcal J}={\mathcal L}$. As already clarified, this means that $I$ is of linear type. \qed
\end{proof}

\section{More applications: Ulrich ideals and a question of Corso-Polini-Rossi}\label{Ulrich}

This last section provides  applications concerning the theory of generalized Ulrich ideals. As we shall explain, this includes a negative answer (in dimension 2) to a question by Corso, Polini, and Rossi.

Here is the central notion (and the setup) of this section.

\begin{Definition}$($\cite[Definition 1.1]{Goto-Ozeki-Takahashi-Watanabe-Yoshida}$)$ \label{Ulrideal} \rm Let $(R, \M)$ be a $d$-dimensional Cohen-Macaulay local ring and let $I$ be an $\M$-primary ideal. Suppose $I$ contains a parameter ideal $J$ as a reduction. Note this is satisfied whenever $R/\mathfrak{m}$ is infinite (indeed, in this situation, a minimal reduction -- which does exist -- of an $\M$-primary ideal must necessarily be a parameter ideal). We say that $I$ is a \textit{generalized Ulrich ideal} -- {\it Ulrich ideal}, for short -- if $I$ satisfies the following properties:
	\begin{itemize}
		\item[(i)] ${\rm r}_{J}(I) \leq 1$;
		\item[(ii)] $I/I^2$ is a free $R/I$-module.
	\end{itemize}
\end{Definition}

\begin{Remark}\label{G-is-CM}\rm Let us recall a couple of interesting basic facts about Ulrich ideals. First, since $R$ is Cohen-Macaulay, it is clear that every parameter ideal is Ulrich. Second, if $R$, $I$ and $J$ are as above and $R/\mathfrak{m}$ is infinite, then by \cite[Lemma 1 and Theorem 1]{Valla} the associated graded ring $\mathcal{G}(I)$ must be Cohen-Macaulay whenever ${\rm r}_{J}(I) \leq 1$. Therefore, if $I$ is Ulrich then $\mathcal{G}(I)$ is Cohen-Macaulay.
\end{Remark}

\subsection{Regularity of blowup algebras} In order to determine the Castelnuovo-Mumford regularity of the blowup algebras of an Ulrich ideal, we first recall one of the ingredients.

\begin{Lemma}$($\cite[(1.3)]{Heinzer-Lantz-Shah}$)$\label{rednum1} If $(R, \M)$ is a Cohen-Macaulay  local ring of positive dimension and $I$ is an $\M$-primary ideal with ${\rm r}(I)\leq 1$, then every power of $I$ is  Ratliff-Rush closed.
\end{Lemma} 

This lemma immediately gives the following fact, which will be useful in the proof of Corollary \ref{reg=red/ulrich}.

\begin{Corollary}\label{UlrRatliff}
	Let $R$ be a Cohen-Macaulay  local ring of positive dimension with infinite residue field and let $I$ be an Ulrich ideal of $R$. Then $$\widetilde{I^{n}}=I^n, \ \ \ \ \ \forall \, n\geq 1.$$ Therefore, for any parameter ideal $J$ which is a reduction of $I$, we have either ${\rm r}_{J}(I) = 0$ or $s^{*}(I)  =  {\rm r}_{J}(I)  = 1$.
\end{Corollary}

This corollary
is particularly useful to test for ideals that are not Ulrich, as illustrated in the next two examples.

\begin{Example} \rm Let $k$ be an infinite field and $R = k[[x,y]]$. Then, by Corollary \ref{UlrRatliff}, the ideal $I = (x^{4}, x^{3}y, xy^{3}, y^{4})$ is not Ulrich, since $x^{2}y^{2} \in \widetilde{I} \setminus I$.
\end{Example}

\begin{Example} \rm Let $k$ be an infinite field and $R = k[[t^{3}, t^{10}, t^{11}]]$. As observed in \cite[Example 1.16]{Heinzer-Lantz-Shah}, the ideal $I = (t^{9}, t^{10}, t^{14})$ is not Ratliff-Rush closed (precisely, $t^{11} \in \widetilde{I} \setminus I$). Hence, Corollary \ref{UlrRatliff} gives that $I$ is not Ulrich.
\end{Example}

The next example shows that the converse of Corollary \ref{UlrRatliff} is not true.

\begin{Example}\rm Let $k$ be an infinite field and consider the zero-dimensional ideal $$I \, = \, (x^{6},\, x^{4}y^{2},\,
	x^{3}y^{3},\, x^{2}y^{4},\, xy^{5},\, y^{6}) \, \subset \, R \, = \, k[[x, y]].$$ According to \cite[Example 2.15]{Mafi2}
	we have $\mathrm{depth}\,\mathcal{G}(I) = 1$ and
	hence $\mathcal{G}(I)$ is not Cohen-Macaulay. By Remark \ref{G-is-CM}, the ideal $I$ cannot be Ulrich. On the other hand, because $\mathrm{grade}\,\mathcal{G}(I)_{+} > 0$ we have $\widetilde{I^{n}} = I^{n}$ for all $n \geq 1$.
\end{Example}

Now we are able to find the regularity of the Rees algebra -- hence of the associated graded ring -- of an Ulrich ideal (in arbitrary positive dimension). 

\begin{Corollary}\label{reg=red/ulrich}
	Let $R$ be a Cohen-Macaulay local ring of positive dimension with infinite residue field, and let $I$ be an Ulrich ideal of $R$. Then,
	$$\mathrm{reg}\,\mathcal{R}(I) \, = \, \mathrm{reg}\,\mathcal{G}(I) \, \leq \, 1,$$ with equality if and only if $I$ is not a parameter ideal. Furthermore, ${\rm r}(I)$ is independent.
\end{Corollary}	
\begin{proof} Pick a minimal reduction $J=(x_1, \ldots, x_d)$ of $I$ (notice that $J$ is necessarily a parameter ideal; see Definition \ref{Ulrideal}). Since $I$ is Ulrich, \cite[Lemma 2.3]{Goto-Ozeki-Takahashi-Watanabe-Yoshida} yields that $I/J$ is a free $R/I$-module. For $i = 1, \ldots, d$, set $R_i = R/(x_1, \ldots, x_{i - 1})$ (with $R_1=R$), $I_i = IR_i$ and $J_i = JR_i$. As $JI = I^2$, we have $J_iI_i = (I_i)^{2}$. Moreover, $I_i/J_i \cong I/J$ and $R_i/I_i \cong R/I$, which gives that $I_i/J_i$ is a free $R_i/I_i$-module. Applying \cite[Lemma 2.3]{Goto-Ozeki-Takahashi-Watanabe-Yoshida} once again, we get that $I_i$ is an Ulrich ideal of $R_i$ for all $i$. Now, by Corollary \ref{UlrRatliff} (with $n=1$) we have $\widetilde{I_i} = I_i=IR_i$ for all $i$. On the other hand, it is easy to see that $\widetilde{I_i} = \widetilde{I_{R_i}}$ in the notation of Theorem \ref{Mafigene} (with $M=R$), which therefore gives $\mathrm{reg}\,\mathcal{R}(I)  =  \mathrm{reg}\,\mathcal{G}(I) = {\rm r}_J(I) \leq  1$, as asserted. Observe that this also shows the independence of ${\rm r}(I)$. 

Now the characterization of equality can be rephrased as ${\rm r}_J(I)=0$ if and only if $I$ is a parameter ideal. Obviously, 
${\rm r}_J(I)=0$ means $I=J$, which is a parameter ideal. Conversely, if $I$ is a parameter ideal then ${\rm r}(I)=0$ (see Subsection \ref{lintype}) and hence ${\rm r}_J(I)=0$ by independence. \qed

%because $R$ is Cohen-Macaulay, $I$ can be %generated by a regular sequence and is thus %{\it basic} in the sense of admitting no %proper reduction (see \cite[Exercise %8.13]{Swanson-Huneke}); in particular, ${\rm %r}_J(I)=0$. 

%It is now clear that $I$ is a parameter %ideal if and only if $I=J$, i.e. ${\rm %r}_{J}(I)=0$. The second assertion %follows.
\end{proof}

\begin{Example} \label{ex-Goto}  \rm Let $k$ be an infinite field. Given $\ell \geq 2$, consider the zero-dimensional ideal $$I \, = \, (t^{4},\, t^{6}) \, \subset \, R \, = \, k[[t^{4}, t^{6}, t^{4\ell - 1}]],$$ which clearly is not a parameter ideal. According to \cite[Example 2.7(1)]{Goto-Ozeki-Takahashi-Watanabe-Yoshida}, this ideal is Ulrich. Applying Corollary \ref{reg=red/ulrich}, we obtain $\mathrm{reg}\,\mathcal{R}(I)  = \mathrm{reg}\,\mathcal{G}(I) = 1$.
\end{Example}

\subsection{Hilbert-Samuel polynomial and postulation number}

Our next goal is to determine the Hilbert-Samuel coefficients and the postulation number of an Ulrich ideal. Recall that if $(R, \M)$ is a Cohen-Macaulay local ring of dimension $d\geq 1$ and $I$ is an $\M$-primary ideal of $R$, then it is a well-known fact that the Hilbert-Samuel polynomial of $I$ can be expressed as 
\begin{equation}\label{hilbpoly}
	P_{I}(n) \, = \, 
	\sum_{i = 0}^{d}(-1)^{i}{\rm e}_{i}(I)\binom{n + d - i - 1}{d - i},
\end{equation} 
where ${\rm e}_{0}(I), \ldots, {\rm e}_{d}(I)$ are the so-called {\it Hilbert-Samuel coefficients of $I$}. The number ${\rm e}_{0}(I)$ (which is always positive) is the multiplicity while ${\rm e}_{1}(I)$ is dubbed {\it Chern number} of $I$. 

\smallskip

For the connection between postulation and reduction numbers, the following fact will be useful.

\begin{Lemma}$($\cite[Theorem 2]{Marley}$)$\label{LemmaM}
Let $(R, \M)$ be a $d$-dimensional Cohen-Macaulay local ring with infinite residue field, and let $I$ be an $\M$-primary ideal with $\mathrm{grade}\,\mathcal{G}(I)_{+} \geq d-1$. Then ${\rm r}(I) = \rho(I) +d$.
\end{Lemma}

\begin{Proposition}\label{HS-poly}
	Let $R$ be a Cohen-Macaulay local ring of dimension $d\geq 1$ and with infinite residue field, and let $I$ be an Ulrich ideal of $R$ minimally generated by $\nu(I)$ elements. Then $$P_{I}(n) \, = \, \lambda(R/I)\left[(\nu(I) - d + 1)\binom{n+d-1}{d} - (\nu(I) - d)\binom{n+d-2}{d-1}\right].$$ Furthermore, $\rho(I) = - d$ if $I$ is a parameter ideal, and $\rho(I) = 1 - d$ otherwise. In particular, $H_I(n)=P_I(n)$ for all $n\geq 1$.
	\end{Proposition}
\begin{proof} Let $J$ be a minimal reduction of $I$ (note $J$ is a parameter ideal; see Definition \ref{Ulrideal}). Since $I^2=IJ$, \cite[Theorem 2.1]{Huneke} gives that the Chern number of $I$ is given by ${\rm e}_{1}(I) = {\rm e}_{0}(I) - \lambda(R/I)$ and in addition ${\rm e}_{j}(I)  =  0$ for all $j = 2, \ldots, d$. Moreover, by \cite[Lemma 1]{Valla},
\begin{equation}\label{vallaequality}
	\lambda(I/I^2) \, = \, {\rm e}_0(I) + (d-1)\lambda(R/I).
\end{equation} 
On the other hand, $I/I^2$ is a free $R/I$-module by hypothesis, and clearly the minimal number of generators of $I/I^2$ coincides with $\nu := \nu(I)$. Thus, setting $\lambda :=\lambda(R/I)$, we get $\lambda(I/I^2)=\lambda((R/I)^{\nu})=\nu \lambda$. Therefore, by (\ref{vallaequality}), $${\rm e}_0(I) \, = \, \nu \lambda - (d-1)\lambda \, = \, \lambda(\nu - d +1)$$ and hence ${\rm e}_1(I) = {\rm e}_0(I)  - \lambda = \lambda(\nu - d)$. Using (\ref{hilbpoly}), the formula for $P_I(n)$ follows. 

Finally,  note that \cite[Theorem 1]{Valla} yields $\mathrm{grade}\,\mathcal{G}(I)_{+} = d$. It now follows from Lemma \ref{LemmaM} that $\rho(I) =  {\rm r}(I) - d$. By Corollary \ref{reg=red/ulrich} and its proof, we get 
${\rm r}(I) \leq 1$ with ${\rm r}(I) = 1$ if and only if $I$ is not a parameter ideal, so the assertions about $\rho(I)$ hold. In particular, $\rho(I)\leq 0$, which gives $H_I(n)=P_I(n)$ for all $n\geq 1$.
\qed 
\end{proof}

\medskip

In the Gorenstein case, Proposition \ref{HS-poly} admits the following version.

\begin{Corollary}\label{HS-poly-Gor}
	Let $R$ be a Gorenstein local ring of dimension $d\geq 1$ and with infinite residue field, and let $I$ be an Ulrich ideal of $R$ which is not a parameter ideal. Then $$P_{I}(n) \, = \, \lambda(R/I^n) \, = \, \frac{\lambda(R/I)}{d!}\cdot \frac{(2n+d-2)(n+d-2)!}{(n-1)!}, \ \ \ \ \ \forall n\geq 2-d.$$ 
	%In particular, $H_I(n)=P_I(n)$ for all %$n\geq 1$.
	\end{Corollary}
\begin{proof} In this case, we have $\nu(I)=d+1$ by \cite[Corollary 2.6(b)]{Goto-Ozeki-Takahashi-Watanabe-Yoshida}, and thus Proposition \ref{HS-poly} gives $P_{I}(n) =  \lambda(R/I)N$, where $N=2\binom{n+d-1}{d} - \binom{n+d-2}{d-1}$. Now, by elementary simplifications, we see that $N=(2n+d-2)(n+d-2)!/d!(n-1)!$, as needed. The fact that the equality holds for all $n\geq 2-d$ follows from $\rho(I)=1-d$, as shown in the proposition. \qed
\end{proof}

\medskip

Next we remark that Ulrich ideals that are not parameter ideals are the same as Ulrich ideals with non-zero Chern number.

\begin{Remark}\rm Maintain the setting and hypotheses of Proposition \ref{HS-poly}. We have shown in particular that the Chern number of $I$ is given by $${\rm e}_1(I) \, = \, \lambda(R/I)(\nu(I) - d).$$ As a consequence, the Ulrich ideal $I$ is a parameter ideal if and only if  ${\rm e}_1(I)=0$. For instance, if $R$ is Gorenstein and $I$ is not a parameter ideal, then ${\rm e}_1(I)=\lambda (R/I)\neq 0$.
\end{Remark}

\begin{Example} \rm Let us revisit the Ulrich ideal $I\subset R$ of Example \ref{ex-Goto}. Note that $\nu(I)=2$ and $\lambda(R/I)=2$. Using Proposition \ref{HS-poly} we get ${\rm e}_0(I)=4$, ${\rm e}_1(I)=2$, $\rho(I)=0$, and then $$P_I(n) \, = \, \lambda(R/I^n) \, = \, 4n-2, \ \ \ \ \ \forall n\geq 1.$$
\end{Example}

\begin{Example}\label{E8} \rm Consider the so-called $E_8$-singularity $R={\mathbb C}[[x, y, z]]/(x^3+y^5+z^2)$, which is a rational double point.
By \cite[Example 7.2]{Goto-et-al-2}, the ideal $I=(x, y^2, z)R$ is an Ulrich ideal. Here we have $d=2$, $\nu(I)=3$ and $\lambda(R/I)=2$, so that ${\rm e}_0(I)=4$, ${\rm e}_1(I)=2$. Also, $\rho(I)=-1$. By Corollary \ref{HS-poly-Gor}, we get
$$P_I(n) \, = \, \lambda(R/I^n) \, = \, \lambda(R/I)n^2 = \, 2n^2, \ \ \ \ \ \forall n\geq 0.$$
\end{Example}

\begin{Example} \rm Fix integers $a\geq b\geq c\geq 2$, and consider the local ring 
$$R \, = \, {\mathbb C}[[x, y, z, t]]/(xy - t^{a+b},\, xz - t^{a+c} + zt^a,\, yz - yt^c+zt^b)$$
which is a rational surface singularity (more precisely, a rational triple point), hence Cohen-Macaulay. Given an integer $\ell$ with $1\leq \ell \leq c$, consider the ideal $I  =  (x, y, z, t^{\ell})R$.  By \cite[Example 7.5]{Goto-et-al-2}, the ideal $I$ is Ulrich, with $\lambda(R/I)=\ell$. Moreover, $d=2$ and $\nu(I)=4$, so that ${\rm e}_0(I)=3\ell$ and ${\rm e}_1(I)=2\ell$. We also have $\rho(I)=-1$. Proposition \ref{HS-poly} thus yields 
$$P_I(n) \, = \, \lambda(R/I^n) \, = \, 3\ell\binom{n+1}{2} - 2\ell n \, = \, \frac{\ell n}{2}(3n-1), \ \ \ \ \ \forall n\geq 0.$$
\end{Example}

\begin{Example} \rm Fix integers $m\geq 1$, $d\geq 2$  and $n_1,\ldots, n_d\geq 2$. Given an infinite field $K$, consider the $d$-dimensional local hypersurface ring $$R \, = \, K[[x_0, x_1, \ldots, x_d]]/(x_0^{2m}+x_1^{n_1}+ \ldots + x_d^{n_d}).$$ Now, fix integers $k_1, \ldots, k_d$ with $1\leq k_i\leq 
\lfloor n_i/2\rfloor$ for all $i=1,\ldots, d$.  By \cite[Example 2.4]{Goto-et-al-2}, the ideal $I  =  (x_0^m, x_1^{k_1}, \ldots, x_d^{k_d})R$ is Ulrich. In this example we have $\lambda(R/I)=mk_1\cdots k_d$, so that 
${\rm e}_0(I)=2mk_1\cdots k_d$ and ${\rm e}_1(I)=mk_1\cdots k_d$. By Corollary \ref{HS-poly-Gor},
$$P_I(n) \, = \, \lambda(R/I^n) \, = \, \frac{mk_1\cdots k_d}{d!}\cdot \frac{(2n+d-2)(n+d-2)!}{(n-1)!}, \ \ \ \ \ \forall n\geq 2-d.$$
In particular, for $d=3$, we have $P_I(n)=\lambda(R/I^n)=\frac{mk_1k_2k_3}{6}(2n+1)(n+1)n$, for all $n\geq -1$. Note that $P_I(-2)=-mk_1k_2k_3\neq 0=H_I(-2)$, while $P_I(-1)=0=H_I(-1)$, $P_I(0)=0=H_I(0)$, $P_I(1)=mk_1k_2k_3=\lambda(R/I)=H_I(1)$, and so on.
\end{Example}

%$$P_I(n) \, = \, \lambda(R/I^n) \, = \, 2mk_1\cdots %k_d\binom{n+d-1}{d} -  mk_1\cdots %k_d\binom{n+d-2}{d-1}, \ \ \ \ \ \forall n\geq %2-d.$$

%\medskip

\subsection{Further results}

In this part we shall provide a correction of a proposition of Mafi as well as improvements of independent results by other authors.

Let $(R, \M)$ be a two-dimensional Cohen-Macaulay local ring with infinite residue field, and let $I$ be an $\M$-primary ideal satisfying $\widetilde{I}=I$. Let $J$ be a minimal reduction of $I$. Then,  \cite[Proposition 2.6]{Mafi2} claims that ${\rm r}_{J}(I) = 2$ if and only if $$H_{I}(n) = P_{I}(n), \ \ \ \ \ n = 1,\, 2.$$ However, if we take $I$ as being an Ulrich ideal, then ${\rm r}_{J}(I) \leq 1$ and we have seen in Corollary \ref{UlrRatliff} that $\widetilde{I}=I$; moreover, our Proposition \ref{HS-poly} yields in particular $H_{I}(n) = P_{I}(n)$ for $n = 1, 2$. Any such $I$ is therefore a counter-example to Mafi's claim. 

We shall establish the correct statement in Proposition \ref{genItoh}. First, we need a lemma.

\begin{Lemma}\label{2lema} Let $(R, \M)$ be a two-dimensional Cohen-Macaulay local ring with infinite residue field, $I$ an $\M$-primary ideal and $J$ a minimal reduction of $I$ with ${\rm r}_{J}(I) \leq 2$. Then, $\widetilde{I}=I$ if and only if $\mathrm{grade}\,\mathcal{G}(I)_{+} \geq 1$.
	\end{Lemma}
\begin{proof} From \cite[Fact 9]{Heinzer-Johnson-Lantz-Shah} we have $\mathrm{grade}\,\mathcal{G}(I)_{+} \geq 1$ if and only if all powers of $I$ are Ratliff-Rush closed. In particular, if $\mathrm{grade}\,\mathcal{G}(I)_{+} \geq 1$ then $\widetilde{I}=I$. Conversely, suppose  $\widetilde{I}=I$.  First, if  ${\rm r}_{J}(I) \leq 1$ then ${\rm r}(I) \leq 1$ and it follows from Lemma \ref{rednum1} that $\widetilde{I^{n}}=I^n$ for all $n \geq 1$. Now if ${\rm r}_{J}(I) = 2$, with say $J=(x, y)$ (note that by letting $M=R$ in Lemma \ref{LemmaRV}, or alternatively by \cite[Lemma 1.2]{Rossi-Dinh-Trung}, we can take $\{x, y\}$ as being a superficial sequence for $I$), then using Proposition \ref{regelem}(b) we can write $$I \, \subseteq \, I^{2} : x \, \subseteq \, \widetilde{I^{2}} : x \, = \, \widetilde{I} \, = \, I,$$ which gives $I^2 : x = I$. By  \cite[Corollary 2.3]{Mafi2} (which requires ${\rm r}_{J}(I) = 2$) we get $\widetilde{I^{n}}=I^n$ for all $n \geq 1$. Therefore, $\mathrm{grade}\,\mathcal{G}(I)_{+} \geq 1$ in both cases. \qed
\end{proof}

\begin{Proposition}\label{genItoh}
Let $(R, \M)$ be a two-dimensional Cohen-Macaulay local ring with infinite residue field, $I$ an $\M$-primary ideal with $\widetilde{I} = I$ and $J$ a minimal reduction of $I$. Then the following assertions are equivalent:
	\begin{itemize}
		\item[(a)] ${\rm r}_{J}(I) \leq 2$;
		\item[(b)] $H_{I}(n) = P_{I}(n)$ for all $n \geq 1$;
		\item[(c)] $H_{I}(n) = P_{I}(n)$ for $n = 1, 2$.
	\end{itemize}
\end{Proposition}
\begin{proof} Assume (a). By Lemma \ref{2lema} we  get $\mathrm{grade}\,\mathcal{G}(I)_{+} \geq 1$, which as we know is equivalent to $\widetilde{I^{n}} = I^{n}$ for all $n \geq 1$. Then (c) follows from \cite[Proposition 16]{Itoh}, which also gives the implication (c) $\Rightarrow$ (a). Thus (a) and (c) are equivalent. The implication (b) $\Rightarrow$ (c) is obvious. Finally let us show that (a) $\Rightarrow$ (b). If ${\rm r}_{J}(I) \leq 2$ then Lemma \ref{2lema} yields $\mathrm{grade}\,\mathcal{G}(I)_{+} \geq 1=d-1$, so we can apply Lemma \ref{LemmaM} and obtain ${\rm r}(I)  =  \rho(I) + 2$. Therefore, $\rho(I)  =  {\rm r}(I) - 2  \leq  {\rm r}_J(I) - 2  \leq   2 - 2  =  0$, which gives (b). \qed
\end{proof}

\begin{Remark}\rm Besides correcting  Mafi's proposition (as explained above), our Proposition \ref{genItoh} also sharpens independent results by Hoa,  Huneke, and Itoh (see \cite[Theorem 3.3]{Hoa}, \cite[Theorem 2.11]{Huneke}, and \cite[Proposition 16]{Itoh}, respectively), where additional hypotheses are required.
\end{Remark}

\subsection{Negative answer to a question of Corso, Polini, and Rossi} In this last part, recall that the integral closure of an ideal $I$ of a Noetherian ring is the set $\overline{I}$ formed by the elements $r$ satisfying an equation of the form $$r^m+a_1r^{m-1}+\ldots + a_{m-1}r+a_m=0, \quad a_i\in I^i, \quad i=1, \ldots, m.$$ Clearly, $\overline{I}$ is an ideal containing $I$. The ideal $I$ is integrally closed (or complete) if $\overline{I}=I$. Moreover, $I$ is
{\it normal} if $I^j$ is integrally closed for every $j\geq 1$. 

Let $(R, \M)$ be a Gorenstein local ring (with positive dimension and infinite residue field), and let $I$ be a normal $\M$-primary ideal of $R$. Then following problem appeared in \cite[Question 4.4]{CPR}: Does ${\rm e}_3(I)=0$ imply ${\rm r}(I)=2$\,? It has been also pointed out that an affirmative answer to this question is given in \cite{Itoh} in the case $I=\mathfrak{m}$ (recall that $\mathfrak{m}$ may not be normal in general). 

Now note that if ${\rm dim}\,R=2$ then the condition ${\rm e}_3(I)=0$ holds trivially. So, the 2-dimensional case of the above question is: If $I$ is a normal $\M$-primary ideal of a 2-dimensional Gorenstein local ring, is it true that ${\rm r}(I)=2$\,? Below we
answer this question negatively. In our example, the ideal $I$ satisfies in addition the property ${\rm e}_2(I)=0$ (indeed, $I$ is Ulrich and this vanishing was observed in the proof of Proposition \ref{HS-poly}). Thus we may also consider a version of the Corso-Polini-Rossi problem by adding the hypothesis ${\rm e}_2(I)\neq 0$, which we hope to pursue in a future work. 

%For a few specific cases, see  \cite{Itoh}.

\begin{Proposition} There exists a normal $\M$-primary ideal $I$, in a $2$-dimensional local hypersurface ring, such that ${\rm r}(I)=1$.
\end{Proposition}
\begin{proof} Consider the rational double point $R={\mathbb C}[[x, y, z]]/(x^3+y^5+z^2)$ and the ideal $I=(x, y^2, z)R$. As mentioned in Example \ref{E8}, the ideal $I$ is Ulrich. As such, being in addition a non-parameter ideal, it satisfies ${\rm r}(I)=1$ (see the proof of Corollary \ref{reg=red/ulrich}). It remains to show that $I$ is normal. First, notice that $I$ is integrally closed, because
$x^3+z^2\notin I^5$ and hence $y\notin \overline{I}$. Now, since $R$ is a rational surface singularity, the normality of $I$ follows by \cite[Theorem 7.1]{Lip}. 
\qed
\end{proof}

\bigskip

\noindent{\bf Acknowledgements.} The first author was partially supported by the CNPq-Brazil grants 301029/2019-9 and 406377/2021-9. The second author was supported by a CAPES Doctoral Scholarship.

\end{document}